\documentclass[12pt,twoside]{article}

\usepackage{graphicx,graphics, amsthm, amsfonts, amsbsy,amssymb, amsmath, cite, array, hyperref, float,tikz,fancyvrb, hyperref}
\usepackage[T1]{fontenc}
\everymath{\displaystyle}
\usepackage[margin=1in]{geometry}

\allowdisplaybreaks
\renewenvironment{proof}{{\noindent\bfseries Proof.}}{\qed}

\newcommand{\spec}{\operatorname{Spec}}
\newcommand{\Tr}{\operatorname{Tr}}

\newtheorem{theorem}{Theorem}[section]
\newtheorem{remark}[theorem]{Remark}
\newtheorem{example}[theorem]{Example}

\newtheorem{corollary}[theorem]{Corollary}

\newtheorem{proposition}[theorem]{Proposition}

\usetikzlibrary{chains,positioning}
\usepackage{tikz,pgfplots}
\usetikzlibrary{decorations.pathmorphing}
\providecommand{\keywords}[1]
{
	\noindent\textbf{Keywords:} #1
}
\providecommand{\ams}[2]
{
	\noindent\textbf{2020 AMS subject classification:} #2
}

\def\qed{\nolinebreak\hfill\rule{.2cm}{.2cm}\par\addvspace{.5cm}}

\begin{document}
\title{Degree-based weighted adjacency matrices: spectra, integrality, and edge deletion effects}
\author{
	{\small Bilal Ahmad Rather$^{a}$ and Hilal Ahmad Ganie$^{b}$} \\[2mm]
	\small $^{a}$School of Mathematics and Statistics, Shandong University of Technology, Zibo 255049, China\\
	\texttt{bilalahmadrr@gmail.com}\\
	\small $^{b}$Department of School  Education JK Govt, Kashmir, India\\
	\texttt{hilahmad1119kt@gmail.com}
}

\date{}

\pagestyle{myheadings} \markboth{Bilal and Hilal}{Degree-based weighted adjacency matrices: spectra, integrality, and edge deletion effects}
\maketitle
\begin{abstract}
	We study degree-based weighted adjacency matrices associated with symmetric edge-weight functions $\phi(d_u,d_v)$, with emphasis on complete multipartite graphs and spectral changes caused by edge modification. We characterize its families with three distinct eigenvalues and identifies integral matrices. For complete graphs, an exact threshold is obtained that determines whether deleting one edge increases, preserves, or decreases both the spectral radius and the weighted energy; the generalized Randić family is classified completely, thereby correcting and refining earlier published results in [Bilal and Munir, Int. J. Quantum Chem. (2024)]. 
	 We further determine the $ISI$ spectrum under edge deletion from regular complete multipartite graphs, derive the complete weighted spectrum, energy, and inertia of crown multipartite graphs, and prove that adding an edge between two leaves of $S_n$ strictly increases its $ISI$ energy.
\end{abstract}

\vskip 3mm


\ams{}{05C50, 05C09, 05C92, 15A18.} \textbf{ACM class:} F.2.2

\section{Introduction}\label{section}

Spectral graph theory provides a natural meeting point between combinatorial structure and linear algebra. Given a finite simple graph $G$, one associates with it a matrix whose entries encode adjacency, incidence, distance, degree, or some other structural relation, and then studies the eigenvalues of that matrix in order to recover information about the graph. The adjacency spectrum is the classical example; its role in graph structure, regularity, connectivity, multipartite decompositions, and extremal problems is well documented in \cite{cds,ping,shi}. A related invariant is graph energy, introduced by Gutman \cite{gutmanenergy} as the absolute sum of the adjacency eigenvalues. Although the original motivation came from the total $\pi$-electron energy of conjugated molecules, graph energy has developed into an independent branch of spectral graph theory and matrix analysis; see, for example, \cite{shi,nv}. In the language of matrix theory, the energy of a real symmetric matrix is its trace norm.

A useful extension of the ordinary adjacency matrix is obtained by assigning a degree-dependent weight to every edge. Let
$
G=(V(G),E(G))
$
be a finite simple undirected graph with
$
V(G)=\{v_1,v_2,\ldots,v_n\},
$ and edge set $E(G).$
The degree of a vertex $v_i$ is denoted by $d_i=d_G(v_i)$, and $u\sim v$ means that $u$ and $v$ are adjacent. Let
$$
\phi:\mathbb R_{>0}\times\mathbb R_{>0}\longrightarrow\mathbb R
$$
be a symmetric function, so that
$$
\phi(x,y)=\phi(y,x).
$$
The associated vertex-degree-based index has the form
$$
\Phi(G)=\sum_{uv\in E(G)}\phi(d_u,d_v).
$$
The corresponding degree-based weighted adjacency matrix, also called a general extended adjacency matrix, is
$$
\dot A(G)=A_{\phi}(G)=\big(a_{ij}^{\phi}\big)_{n\times n},
$$
where
$$
a_{ij}^{\phi}=
\begin{cases}
	\phi(d_i,d_j),&v_i\sim v_j,\\
	0,&v_i\not\sim v_j.
\end{cases}
$$
This matrix formalism is valuable because a scalar topological index records only the total contribution of the edge weights, whereas $A_{\phi}(G)$ retains their distribution over the graph and therefore allows eigenvalue, rank, nullity, inertia, integrality, and energy questions to be studied simultaneously. The general degree-based energy framework was developed systematically in\cite{dasenergy}; related bounds for the extended adjacency spectral radius and energy were obtained in \cite{dgf,glza,gutman2022}. Weighted graph energy in a broader setting was considered in \cite{gs,HB}.

Several familiar matrices arise from particular choices of $\phi$. The cases used throughout this paper are summarized below:
$$
\begin{array}{c|c|c}
	\text{notation}&\phi(x,y)&\text{matrix or weight}\\
	\hline
	A&1&\text{ordinary adjacency}\\
	ISI&\frac{xy}{x+y}&\text{inverse sum indeg}\\
	AG&\frac{x+y}{2\sqrt{xy}}&\text{arithmetic--geometric}\\
	GA&\frac{2\sqrt{xy}}{x+y}&\text{geometric--arithmetic}\\
	M1&x+y&\text{first Zagreb}\\
	ABC&\sqrt{\frac{x+y-2}{xy}}&\text{atom--bond connectivity}\\
	R&\frac{1}{\sqrt{xy}}&\text{Randi'c}\\
	M2&xy&\text{second Zagreb}\\
	S&\sqrt{x^2+y^2}&\text{Sombor}\\
	MS&\frac{1}{\sqrt{x^2+y^2}}&\text{modified Sombor}.
\end{array}
$$
We also use the generalized Randić family
$\phi_{\alpha}(x,y)=(xy)^{\alpha},$ for $ \alpha\in\mathbb R.$
These matrices have been studied from different spectral points of view. The $ABC$ matrix and its energy were investigated by Chen \cite{chen2018}; arithmetic--geometric and geometric--arithmetic spectral quantities were treated in \cite{guo,zheng,rodriguez2015,lin1}; Sombor energy and related inequalities appear in \cite{gutman2021,gutmancontributions,bilalMatch,bilalMatch1}; and degree-based random weighted matrices were considered in \cite{li2,li3}. General extended adjacency eigenvalues of structured graph families, including chain graphs, were further investigated in \cite{bilalmdpi}. For the inverse sum indeg matrix in particular, recent work has developed bounds and structural properties of its spectrum and energy \cite{Jamal2023}.
The eigenvalues $\lambda_1(\dot A(G)) $
is positive and simple and admits a strictly positive eigenvector. This positivity is important in the spectral-radius comparisons used later.

The current literature increasingly treats $A_{\phi}(G)$ through a unified rather than matrix-by-matrix approach. Li and Yang established interlacing results for degree-based weighted adjacency eigenvalues under graph operations such as vertex deletion, contraction, and edge subdivision \cite{LiYang2023}. Gao, Li and Yang studied the response of the principal eigenvalue to perturbations \cite{GaoLiYang2024}, while Gao and Yang obtained further bounds on the largest eigenvalue, including estimates adapted to edge deletion and other graph modifications \cite{GaoYang2024}. Two complementary studies by Li and Yang developed unified conditions for weighted spectral properties, numbers of distinct eigenvalues, extremal spectral radius, and energy \cite{LiYangLAA2024,LiYangActa2024}. More recently, Shen and Shan investigated extremal $f$-spectral radii for graphs with prescribed order and size \cite{ShenShan2025}, Li and Zheng obtained extremal results for function-weighted adjacency matrices \cite{LiZheng2025}, and Gao, Li, Yang and Zheng studied Li--Feng transformations for degree-based weighted adjacency matrices \cite{GaoLiYangZheng2025}. These developments show that graph transformations and structured quotient reductions have become central tools in the subject.

One of the main structural families considered here is the complete multipartite graph
$
K_{p_1,p_2,\ldots,p_t}.
$
Its vertex set is partitioned into $t$ independent sets
$
V_1,V_2,\ldots,V_t,
$ with $ |V_i|=p_i, $
and every two vertices from different parts are adjacent. If
$
n=p_1+p_2+\cdots+p_t,
$
then each vertex in $V_i$ has degree
$
d_i=n-p_i.
$
The ordinary adjacency spectrum of a complete multipartite graph is classical, but the degree-based weighted case contains additional information because edges joining different pairs of parts can carry different weights. The block structure nevertheless allows a substantial reduction. Each part $V_i$ contributes a $(p_i-1)$-dimensional zero-sum subspace on which $\dot A(G)$ vanishes. The remaining spectral information is carried by a $t$-dimensional quotient space.
$
n-t+\dim\ker S.
$
Thus $n-t$ is the forced nullity supplied by the partite classes, but the actual zero multiplicity may be larger when the symmetric quotient is singular. 
This reduction also clarifies the problem of having only a few distinct weighted eigenvalues.  More generally, under nonsingularity of $S$, a complete multipartite graph has exactly three distinct weighted eigenvalues precisely when all off-diagonal entries of the symmetric quotient coincide, that is,
$$
\sqrt{p_ip_j},\phi(d_i,d_j)=c_0
$$
for a common positive constant $c_0$ and every $i\neq j$. Regular complete multipartite graphs satisfy this automatically, but regularity is not necessary for an arbitrary degree weight. This distinction is significant: for the $ISI$ matrix, the irregular family $K_{9b,b,b}$ has exactly three distinct eigenvalues, while for the second Zagreb matrix the same phenomenon occurs for $K_{4b,b,b,b}$. Hence the symmetric quotient, rather than equality of the part sizes alone, is the correct object controlling the three-eigenvalue phenomenon.
$
2(t-1)pw.
$
This formula also provides the correct integrality criterion:
$$
\spec(\dot A(K_p^{(t)}))\subseteq\mathbb Z
\quad\Longleftrightarrow\quad
p\phi(d,d)\in\mathbb Z.
$$
This is sharper than requiring $\phi(d,d)$ itself to be integral. The integrality problem belongs to the broader program initiated by Harary and Schwenk \cite{harary}, who asked which graphs have integral spectra.

The second principal theme is the behavior of spectral radius and energy under local graph modifications. Energy change under edge addition and deletion has a long history. Day and So studied singular-value inequalities and graph-energy change \cite{day,day2}; Akbari, Ghorbani and Oboudi investigated edge addition and energy \cite{akbari}; Gutman and Shao treated weighted energy change \cite{gs}; and Wang and So analyzed the effect of arbitrary single-edge deletion \cite{wang}. For degree-based matrices the problem is subtler than for a fixed weighted graph because deleting an edge changes not only the adjacency pattern but also the degrees of its endpoints, and therefore changes the weights of all remaining edges incident with those vertices. This distinction is essential in assessing earlier claims concerning $ISI$ energy \cite{bilalpak}.
The comparison between $K_n$ and $K_n-e$ is therefore governed exactly by
$
R_n=\tfrac{b_n}{a_n}
$
and the threshold
$
\tau_n=\sqrt{\tfrac{n-2}{n-1}}.
$
Both spectral radius and energy increase if $R_n<\tau_n$, remain unchanged if $R_n=\tau_n$, and decrease if $R_n>\tau_n$. Thus there is no universal monotonicity principle independent of the chosen degree weight. For the generalized Randi'c weight $(xy)^{\alpha}$, the threshold becomes especially transparent: deletion increases both quantities when $\alpha<-\tfrac12$, preserves them when $\alpha=-\tfrac12$, and decreases them when $\alpha>-\tfrac12$. The ordinary Randi'c matrix is therefore the exact equality case.

For the inverse sum indeg matrix, deletion of an edge from a regular graph requires additional care. If $G$ is $r$-regular, then
$
ISI(G)=\frac r2A(G),
$
but this identity no longer applies after an edge is deleted because the resulting graph is not regular. For
$
G_{t,p}=K_p^{(t)}-e
$
with
$
r=p(t-1),
$
the two exceptional vertices have degree $r-1$ and the remaining vertices have degree $r$.  The complete $ISI$ spectrum is given and makes it possible to settle the energy comparison without incorrectly transferring an adjacency-energy statement from the original regular graph to the nonregular deleted graph.

Finally, we compare the star
$
S_n=K_{1,n-1}
$
with the unicyclic graph $S_n^+$ obtained by adding an edge between two leaves. We prove that
$
\dot E_{ISI}(S_n^+)>\dot E_{ISI}(S_n)
$
for every $n\geq4$. Thus adding an edge between two leaves of a star strictly increases its $ISI$ energy.

	The following notations are used in the article: $\rho(M)$ is the spectral radius of $M$, for a connected graph with positive edge weights, $\rho(\dot A(G))=\lambda_1(\dot A(G))$. The multiset of eigenvalues of $M$ is denoted by $\spec(M)$, $\lambda^{[k]}$ denoted the eigenvalue $\lambda$ repeated with multiplicity $k$. The algebraic multiplicity of $\lambda$ as an eigenvalue of $M$ is denoted by $\operatorname{mult}_{M}(\lambda)$, $\chi_M(x)$ is the characteristic polynomial
	$
	\chi_M(x)=\det(xI-M)
	$, $\operatorname{Tr}(M)$, $\det(M)$, $\operatorname{rank}(M)$ are the trace, determinant, and rank of $M$, respectively.
	The nullspace of $M$ and its dimension is denoted by $\ker M$ and $\dim\ker M$,  $\operatorname{supp}(x)$ is the support of a vector $x$, namely the coordinates on which $x$ is nonzero, $e_i$ is the $i$-th standard coordinate vector. $\mathbf 1$ and $\mathbf 1_X$ are respectively an all-ones vector and the indicator/all-ones vector supported on the set $X$. $I_k$, $J_k$, $J_{r\times s}$, $0_k$ are the identity matrix of order $k$, the $k\times k$ all-ones matrix, the $r\times s$ all-ones matrix, and the zero matrix of order $k$. The direct sum of vector spaces or matrices is $\oplus$,  $\otimes$ is the Kronecker product, and  $\uplus$ is the multiset union. The inertia $\operatorname{In}(M)$ of a real symmetric matrix, written as $(n_+,n_-,n_0)$, where $n_+$, $n_-$, and $n_0$ are the numbers of positive, negative, and zero eigenvalues. The matrix $\dot A(G)$ is called Integral weighted matrix when all eigenvalues of $\dot A(G)$ are integers.

The article is organized as follows: The results developed in Sections~\ref{section 2} and \ref{section 3} are therefore organized around two complementary principles. The first is \emph{spectral compression}: equitable partitions, orthogonal zero-sum subspaces, diagonal similarity, tensor products, and symmetry reduce large weighted adjacency matrices to small matrices carrying all nontrivial spectral information. The second is \emph{degree-sensitive perturbation}: when an edge is added or deleted, one must account simultaneously for the changed zero pattern and for the new degrees, because the edge weights themselves are recomputed. Together these ideas give exact spectra, corrected integrality criteria, sharp energy comparisons, and explicit exceptional families.

\section{Weighted spectra of complete multipartite graphs}\label{section 2}

This section develops the spectral reduction for degree-based weighted adjacency matrices of complete multipartite graphs. The reduction itself is robust, but two points require care. First, the usual equitable quotient is generally not symmetric, even though the original weighted adjacency matrix is symmetric. The right object for spectral arguments is a diagonally similar symmetric quotient. Second, the nullity contributed by the partite sets is always at least $n-t$, but it can be larger when the reduced matrix is singular. These observations also lead to a corrected characterization of the three-eigenvalue case and to a sharper integrality criterion.

Throughout this section, $\phi(x,y)=\phi(y,x)$ is assumed to be real and nonzero on degree pairs that occur on edges. Whenever Perron--Frobenius arguments are used, we additionally assume $\phi(x,y)>0$.

\medskip
The following result is concerned for the eigenvalue $0$ of $\dot A(G)$ provided $G$ has special structure.
\begin{theorem}[\cite{bilalmdpi}]\label{independence Sombor eigenvalue}
	Let $G$ be connected and let $S=\{v_1,\ldots,v_\alpha\}$ be an independent set such that $N(v_i)=N(v_j)$ for all $i,j$. Then $0$ is an eigenvalue of $\dot A(G)$ with multiplicity at least $\alpha-1$.
\end{theorem}

Next, result is about the eigenvalue $-\phi(d,d)$ of  $\dot A(G)$, when $G$ has special clique structure.
\begin{theorem}[\cite{bilalmdpi}]\label{clique Sombor eigenvalue}
	Let $G$ be connected and let $S=\{v_1,\ldots,v_\alpha\}$ be a clique such that
	$N(v_i)\setminus\{v_j\}=N(v_j)\setminus\{v_i\}$ for all $i,j$. Then $d_{v_1}=\cdots=d_{v_\alpha}=:d$, and $-\phi(d,d)$ is an eigenvalue of $\dot A(G)$ with multiplicity at least $\alpha-1$.
\end{theorem}

For an equitable partition $\Pi=\{V_1,\ldots,V_t\}$, the ordinary quotient records the constant row sum from a vertex of $V_i$ into $V_j$. In the present setting this quotient is useful computationally, but it need not be symmetric. The next theorem replaces it by a symmetric matrix without changing the reduced spectrum.

\begin{theorem}\label{thm:spectrum_complete_multipartite_updated}
	Let $G\cong K_{p_1,p_2,\ldots,p_t}$, where $t\geq 2$, $p_i\geq 1$,
	$
	n=\sum_{i=1}^t p_i,
	$
	and let
	$
	d_i=n-p_i
	$
	be the common degree of the vertices belonging to the $i$-th partite set. Let
	$
	M=(m_{ij})_{t\times t},
	$
	where
	$
	m_{ii}=0
	$
	and
	$
	m_{ij}=p_j\phi(d_i,d_j)~
	 (i\neq j).	$
	Let
	$
	D=\operatorname{diag}(p_1,\ldots,p_t)
	$
	and define
	$
	S=D^{1/2}MD^{-1/2}=(s_{ij})_{t\times t}.
	$
	Then
	$
	s_{ii}=0
	$
	and, for $i\neq j$,
	$
	s_{ij}=\sqrt{p_ip_j}\,\phi(d_i,d_j).
	$
	In particular, under the standing assumption that $\phi$ is symmetric, $S$ is a real symmetric matrix and $M$ is similar to $S$. Moreover,
	$
	\chi_{\dot A(G)}(x)
	=
	x^{n-t}\det(xI_t-S).
	$
	Consequently,
	$
	\spec(\dot A(G))
	=
	\{0^{[n-t]}\}\uplus\spec(S)
	=
	\{0^{[n-t]}\}\uplus\spec(M),
	$
	where $\uplus$ denotes multiset union. Furthermore,
	$
	\operatorname{rank}\dot A(G)
	=
	\operatorname{rank}S
	$
	and
	$
	\operatorname{mult}_{\dot A(G)}(0)
	=
	n-t+\dim\ker S.
	$
	Hence the multiplicity of $0$ is exactly $n-t$ if and only if $S$ is nonsingular.
\end{theorem}

\begin{proof}
	Let
	$
	V(G)=V_1\sqcup V_2\sqcup\cdots\sqcup V_t,
	~ |V_i|=p_i,
	$
	be the canonical multipartite decomposition of $G$. Since $G$ is complete multipartite, no two vertices belonging to the same partite set are adjacent, whereas every vertex of $V_i$ is adjacent to every vertex of $V_j$ whenever $i\neq j$. Every vertex of $V_i$ has degree
	$
	d_i=n-p_i.
	$
	Therefore all edges joining $V_i$ and $V_j$ receive the common weight
	$
	\phi(d_i,d_j).
	$
	With the vertices ordered according to $V_1,\ldots,V_t$, the matrix $\dot A(G)$ has block form
	$$
	\dot A(G)
	=
	\begin{pmatrix}
		0_{p_1}
		&
		\phi(d_1,d_2)J_{p_1\times p_2}
		&
		\cdots
		&
		\phi(d_1,d_t)J_{p_1\times p_t}
		\\
		\phi(d_2,d_1)J_{p_2\times p_1}
		&
		0_{p_2}
		&
		\cdots
		&
		\phi(d_2,d_t)J_{p_2\times p_t}
		\\
		\vdots
		&
		\vdots
		&
		\ddots
		&
		\vdots
		\\
		\phi(d_t,d_1)J_{p_t\times p_1}
		&
		\phi(d_t,d_2)J_{p_t\times p_2}
		&
		\cdots
		&
		0_{p_t}
	\end{pmatrix}.
	$$
	
	For each $i\in\{1,\ldots,t\}$, define
	$$
	U_i
	=
	\left\{
	x\in\mathbb R^{V(G)}:
	\operatorname{supp}(x)\subseteq V_i,\ 
	\sum_{u\in V_i}x_u=0
	\right\}.
	$$
	The condition that the coordinates of $x$ on $V_i$ sum to zero imposes one independent linear relation on a $p_i$-dimensional coordinate space. Hence
	$
	\dim U_i=p_i-1.
	$
	
	We first show that every vector in $U_i$ belongs to the kernel of $\dot A(G)$. Let $x\in U_i$. If $v\in V_i$, then $x$ is supported entirely on $V_i$, while there are no edges inside $V_i$. Consequently,
	$
	(\dot A(G)x)_v=0.
	$
	Now let $v\in V_j$ with $j\neq i$. Since $v$ is adjacent to every vertex of $V_i$ and all these edges have weight $\phi(d_j,d_i)$,
	$$
	(\dot A(G)x)_v
	=
	\sum_{u\in V_i}\phi(d_j,d_i)x_u
	=
	\phi(d_j,d_i)\sum_{u\in V_i}x_u
	=
	0.
	$$
	Thus every coordinate of $\dot A(G)x$ is zero, and therefore
	$
	\dot A(G)x=0.
	$
	Hence
	$
	U_i\subseteq\ker\dot A(G)
	$
	for every $i$.
	
	The subspaces $U_1,\ldots,U_t$ have pairwise disjoint supports. In particular, they are mutually orthogonal. Therefore their sum is direct, and if
	$$
	U=\bigoplus_{i=1}^t U_i,
	$$
	then
	$
	U\subseteq\ker\dot A(G)
	$
	and
	$$
	\dim U
	=
	\sum_{i=1}^t(p_i-1)
	=
	\sum_{i=1}^t p_i-t
	=
	n-t.
	$$
	This already explains the forced occurrence of at least $n-t$ zero eigenvalues.
	
	To identify the remaining spectral information, define the normalized characteristic vector of $V_i$ by
	$$
	z_i
	=
	\frac{1}{\sqrt{p_i}}\mathbf 1_{V_i},
	\qquad i=1,\ldots,t,
	$$
	and let
	$
	W=\operatorname{span}\{z_1,\ldots,z_t\}.
	$
	Since the sets $V_i$ are pairwise disjoint,
	$
	\langle z_i,z_j\rangle=0$ for $i\neq j,$
	while
	$$
	\|z_i\|^2
	=
	\frac{1}{p_i}\langle\mathbf 1_{V_i},\mathbf 1_{V_i}\rangle
	=
	1.
	$$
	Thus
	$
	z_1,\ldots,z_t
	$
	form an orthonormal basis of $W$, and
	$
	\dim W=t.
	$
	
	Next, we verify that $U$ and $W$ are orthogonal complements. If $x\in U_i$, then
	$$
	\langle x,z_i\rangle
	=
	\frac{1}{\sqrt{p_i}}
	\sum_{u\in V_i}x_u
	=
	0.
	$$
	Moreover, for $j\neq i$, the supports of $x$ and $z_j$ are disjoint, so
	$
	\langle x,z_j\rangle=0.
	$
	Hence
	$
	U\perp W.
	$
	Since
	$
	\dim U+\dim W=(n-t)+t=n,
	$
	we obtain the orthogonal decomposition
	$
	\mathbb R^{V(G)}=U\oplus W.
	$
	
	It remains to determine the action of $\dot A(G)$ on $W$. Fix $j\in\{1,\ldots,t\}$. Since $z_j$ is supported on $V_j$, there is no contribution from the diagonal block corresponding to $V_j$. For $i\neq j$, every coordinate of $\dot A(G)z_j$ belonging to $V_i$ equals
	$$
	\phi(d_i,d_j)
	\sum_{u\in V_j}\frac{1}{\sqrt{p_j}}
	=
	\phi(d_i,d_j)\frac{p_j}{\sqrt{p_j}}
	=
	\sqrt{p_j}\,\phi(d_i,d_j).
	$$
	Therefore the $V_i$-component of $\dot A(G)z_j$ is
	$
	\sqrt{p_j}\,\phi(d_i,d_j)\mathbf 1_{V_i}.
	$
	Since
	$
	\mathbf 1_{V_i}=\sqrt{p_i}\,z_i,
	$
	this can be written as
	$
	\sqrt{p_ip_j}\,\phi(d_i,d_j)z_i.
	$
	Hence
	$$
	\dot A(G)z_j
	=
	\sum_{\substack{i=1\\i\neq j}}^t
	\sqrt{p_ip_j}\,\phi(d_i,d_j)z_i.
	$$
	In particular, $\dot A(G)z_j\in W$, so $W$ is invariant under $\dot A(G)$.
	 Relative to the orthonormal basis
	$
	z_1,\ldots,z_t,
	$
	the matrix of the restriction $\dot A(G)|_W$ is 
	$
	S=(s_{ij})_{t\times t},
	$
	where
	$
	s_{ii}=0
	$
	and
	$$
	s_{ij}
	=
	\sqrt{p_ip_j}\,\phi(d_i,d_j)
	\qquad(i\neq j).
	$$
	Since $\phi$ is symmetric,
	$
	\phi(d_i,d_j)=\phi(d_j,d_i),
	$
	and consequently
	$$
	s_{ij}
	=
	\sqrt{p_ip_j}\,\phi(d_i,d_j)
	=
	\sqrt{p_jp_i}\,\phi(d_j,d_i)
	=
	s_{ji}.
	$$
	Thus $S$ is real symmetric.
	
	Choose an orthonormal basis of each $U_i$, combine these bases to obtain an orthonormal basis of $U$, and append the vectors
	$
	z_1,\ldots,z_t.
	$
	With respect to the resulting orthonormal basis of $\mathbb R^{V(G)}$, the operator $\dot A(G)$ has block diagonal matrix
	$
	0_{n-t}\oplus S.
	$
	Equivalently, there exists an orthogonal matrix $P$ such that
	$$
	P^{\top}\dot A(G)P
	=
	\begin{pmatrix}
		0_{n-t}&0\\
		0&S
	\end{pmatrix}.
	$$
	Thus $\dot A(G)$ is orthogonally similar to $0_{n-t}\oplus S$. It follows that
	$$
	\chi_{\dot A(G)}(x)
	=
	\det\left(xI_n-(0_{n-t}\oplus S)\right).
	$$
	Using the determinant of a block diagonal matrix,
	$$
	\chi_{\dot A(G)}(x)
	=
	\det(xI_{n-t})\det(xI_t-S)
	=
	x^{n-t}\det(xI_t-S).
	$$
	This proves the asserted characteristic-polynomial factorization.
	
	The same orthogonal decomposition gives the spectrum, including multiplicities:
	$$
	\spec(\dot A(G))
	=
	\{0^{[n-t]}\}\uplus\spec(S).
	$$
	Notice that the multiset notation is important here. If $S$ itself is singular, additional zero eigenvalues arise from $\spec(S)$ and must be added to the forced multiplicity $n-t$.
	
	We now connect $S$ with the equitable quotient $M$. Since
	$
	D^{1/2}
	=
	\operatorname{diag}(\sqrt{p_1},\ldots,\sqrt{p_t})
	$
	and
	$$
	D^{-1/2}
	=
	\operatorname{diag}
	\left(
	\frac{1}{\sqrt{p_1}},\ldots,
	\frac{1}{\sqrt{p_t}}
	\right),
	$$
	the $(i,j)$-entry of $D^{1/2}MD^{-1/2}$, for $i\neq j$, is
	$$
	\sqrt{p_i}
	\left(
	p_j\phi(d_i,d_j)
	\right)
	\frac{1}{\sqrt{p_j}}
	=
	\sqrt{p_ip_j}\,\phi(d_i,d_j).
	$$
	The diagonal entries remain zero. Therefore
	$
	D^{1/2}MD^{-1/2}=S.
	$
	As every $p_i$ is positive, $D^{1/2}$ is invertible. Hence $M$ and $S$ are similar, and consequently
	$
	\spec(M)=\spec(S).
	$
	This implies that
	$$
	\spec(\dot A(G))
	=
	\{0^{[n-t]}\}\uplus\spec(S)
	=
	\{0^{[n-t]}\}\uplus\spec(M).
	$$
	
	The rank formula is now immediate from the orthogonal block decomposition. Since
	$
	\dot A(G)\sim 0_{n-t}\oplus S,
	$
	similarity preserves rank and
	$
	\operatorname{rank}(0_{n-t}\oplus S)
	=
	\operatorname{rank}S.
	$
	Therefore
	$
	\operatorname{rank}\dot A(G)
	=
	\operatorname{rank}S.
	$
	
	Likewise,
	$
	\ker(0_{n-t}\oplus S)
	=
	\mathbb R^{n-t}\oplus\ker S.
	$
	Consequently,
	$
	\dim\ker\dot A(G)
	=
	n-t+\dim\ker S.
	$
	Since $\dot A(G)$ is real symmetric, it is diagonalizable, and hence the dimension of its kernel equals the algebraic multiplicity of the eigenvalue $0$. Thus
	$$
	\operatorname{mult}_{\dot A(G)}(0)
	=
	n-t+\dim\ker S.
	$$
	It follows that
	$
	\operatorname{mult}_{\dot A(G)}(0)=n-t
	$
	if and only if
	$
	\dim\ker S=0.
	$
	For a square matrix this is equivalent to
	$
	\det S\neq0,
	$
	that is, to the nonsingularity of $S$. This completes the proof.
\end{proof}

The following example illustrates the above result.
\begin{example}\label{ex:extra_zero_multipartite}
	The term $n-t$ in Theorem~\ref{thm:spectrum_complete_multipartite_updated} need not be the full multiplicity of $0$. Consider $G=K_{4,3,2,1}$, so $n=10$, $t=4$, and the four degrees are $6,7,8,9$. Choose a positive symmetric degree-weight function on the six relevant degree pairs so that
	$
	\sqrt{p_ip_j}\phi(d_i,d_j)=1
	$
	for every pair except the pair joining the parts of sizes $2$ and $1$, for which
	$
	\sqrt{2}\phi(8,9)=4.
	$
	Then the symmetric quotient is
	$$
	S=
	\begin{pmatrix}
		0&1&1&1\\
		1&0&1&1\\
		1&1&0&4\\
		1&1&4&0
	\end{pmatrix},
	$$
	and
	$
	\spec(S)=\{5,0,-1,-4\}.
	$
	Hence
	$$
	\spec(\dot A(G))=\{5,0^{[7]},-1,-4\}.
	$$
	Here $n-t=6$, but the actual zero multiplicity is $7$.
\end{example}

The following consequences are immediate from Theorem \ref{thm:spectrum_complete_multipartite_updated}.
\begin{corollary}\label{cor:bipartite_weighted_spectrum}
	Let $G=K_{p,q}$ with $p,q\ge 1$, $d_1=q$, $d_2=p$ and
	$
	a=\sqrt{pq}\,\phi(p,q).
	$
	Then
	$
	\spec(\dot A(G))=\{0^{[p+q-2]},a,-a\}.
	$
	Consequently,
	$
	\dot E_\phi(G)=2|a|,
	$
	and, when $\phi(p,q)>0$,
	$
	\rho(\dot A(G))=a.
	$
	The weighted spectrum is integral if and only if $a\in\mathbb Z$.
\end{corollary}

\begin{corollary}\label{thm:tripartite_cubic_inertia}
	Let $G=K_{p,q,r}$ and assume that the relevant weights are positive. Let $n=p+q+r$ and
	$ a=\sqrt{pq}\,\phi(n-p,n-q),		b=\sqrt{pr}\,\phi(n-p,n-r),$ and $
		c=\sqrt{qr}\,\phi(n-q,n-r). $
	Then
	$
	\chi_{\dot A(G)}(x)
	=x^{n-3}\left(x^3-(a^2+b^2+c^2)x-2abc\right).
	$
	Moreover, $0$ has multiplicity exactly $n-3$, the three nonzero eigenvalues consist of one positive and two negative numbers, and
	$
	\dot E_\phi(G)=2\rho(\dot A(G)).
	$
	If $n>3$, then $\dot A(G)$ has exactly three distinct eigenvalues if and only if
	$
	a=b=c.
	$
	In that case, writing the common value as $c_0>0$, and
	$
	\spec(\dot A(G))=\{0^{[n-3]},2c_0,(-c_0)^{[2]}\}.
	$
\end{corollary}

\begin{proof}
	The symmetric quotient is
	$$
	S=\begin{pmatrix}
		0&a&b\\
		a&0&c\\
		b&c&0
	\end{pmatrix}.
	$$
	A direct determinant expansion gives
	$$
	\det(xI_3-S)=x^3-(a^2+b^2+c^2)x-2abc.
	$$
	Since
	$
	\det S=2abc>0,
	$
	the matrix $S$ is nonsingular. Since $S$ is irreducible and entrywise positive off the diagonal, its largest eigenvalue is positive and simple. Since $\operatorname{Tr}S=0$ and $\det S>0$, the other two eigenvalues must both be negative. Thus the energy equals twice the unique positive eigenvalue. For the last assertion, the discriminant of the cubic is
	$$
	4\left((a^2+b^2+c^2)^3-27a^2b^2c^2\right).
	$$
	By the arithmetic--geometric mean inequality this quantity is nonnegative, and it vanishes exactly when $a^2=b^2=c^2$. Also, the positivity of $a,b,c$  gives $a=b=c$. In that case $S=c_0(J_3-I_3)$, whose spectrum is $\{2c_0,-c_0,-c_0\}$.
\end{proof}

The following result gives the three distinct eigenvalues of complete multipartite graphs.

\begin{theorem}\label{cor:regular_complete_multipartite_updated}
	Let $G=K_{p_1,\ldots,p_t}$ with $t\geq 3$, $n>t$, and suppose that $\phi(d_i,d_j)>0$ for all $i\neq j$. Assume that the symmetric quotient $S$ in Theorem~\ref{thm:spectrum_complete_multipartite_updated} is nonsingular. Then $\dot A(G)$ has exactly three distinct eigenvalues if and only if there exists a constant $c>0$ such that
	$
	\sqrt{p_ip_j}\,\phi(d_i,d_j)=c
	$
	for every $i\neq j$. In this case,
	$
	\spec(\dot A(G))
	=
	\left\{
	0^{[n-t]},
	((t-1)c)^{[1]},
	(-c)^{[t-1]}
	\right\}.
	$
	Every regular complete multipartite graph satisfies this condition, although the converse need not hold for a general degree-based weight.
\end{theorem}

\begin{proof}
	Recall from Theorem~\ref{thm:spectrum_complete_multipartite_updated} that the nonzero part of the spectrum of $\dot A(G)$ is completely determined by the symmetric quotient
	$
	S=(s_{ij})_{t\times t},$ with $
	s_{ii}=0, 
	s_{ij}=\sqrt{p_ip_j}\,\phi(d_i,d_j)
	$ for $i\neq j.$
	Moreover,
	$
	\spec(\dot A(G))
	=
	\{0^{[n-t]}\}\uplus\spec(S).
	$
	Since $S$ is assumed to be nonsingular, $0\notin\spec(S)$. Consequently, the multiplicity of $0$ as an eigenvalue of $\dot A(G)$ is exactly $n-t$. 
	
	Suppose first that there exists a constant $c>0$ such that
	$
	\sqrt{p_ip_j}\,\phi(d_i,d_j)=c
	$
	for every $i\neq j$. Then every off-diagonal entry of $S$ is equal to $c$, whereas its diagonal entries vanish. Hence
	$
	S=c(J_t-I_t).
	$
	The all-ones vector $\mathbf 1$ satisfies
	$
	(J_t-I_t)\mathbf 1=(t-1)\mathbf 1,
	$
	and therefore $(t-1)c$ is an eigenvalue of $S$. On the other hand, if $x\perp\mathbf 1$, then $J_tx=0$, and hence
	$$
	Sx
	=
	c(J_t-I_t)x
	= 
	-cx.
	$$
	The subspace $\mathbf 1^{\perp}$ has dimension $t-1$, so $-c$ has multiplicity $t-1$. Thus
	$
	\spec(S)
	=
	\left\{
		((t-1)c)^{[1]},
		(-c)^{[t-1]}
		\right\}.
	$
	Combining this with the $n-t$ zero eigenvalues supplied by the partite classes yields
	$
	\spec(\dot A(G))
	=
	\left\{
		0^{[n-t]},
		((t-1)c)^{[1]},
		(-c)^{[t-1]}
		\right\}.
	$
	Since $n>t$, the eigenvalue $0$ is present, and because $c>0$ and $t\geq3$, the three numbers
	$
	0, -c, (t-1)c
	$
	are pairwise distinct. Hence $\dot A(G)$ has exactly three distinct eigenvalues. 
	Conversely, assume that $\dot A(G)$ has exactly three distinct eigenvalues. Since $n>t$, Theorem~\ref{thm:spectrum_complete_multipartite_updated} shows that $0$ is an eigenvalue of $\dot A(G)$. Since $S$ is nonsingular, none of the eigenvalues of $S$ is zero. Therefore the remaining two distinct eigenvalues of $\dot A(G)$ are precisely the two distinct eigenvalues of $S$. We now use the positivity assumption. For $i\neq j$,
	$
	s_{ij}
	=
	\sqrt{p_ip_j}\,\phi(d_i,d_j)>0.
	$
	Thus $S$ is a nonnegative irreducible symmetric matrix. Indeed, its associated graph is the complete graph $K_t$, because every off-diagonal entry is positive. By the Perron--Frobenius theorem, its spectral radius is a positive simple eigenvalue. Denote this eigenvalue by $\lambda$, so that
	$
	\lambda=\rho(S)>0.
	$
	Let $v$ be the corresponding unit Perron eigenvector. We may choose
	$
	v=(v_1,\ldots,v_t)^{\top}
	$
	with
	$
	v_i>0
	$
	for every $i$ and
	$
	v^{\top}v=1.
	$
	
	As $S$ has exactly two distinct eigenvalues and $\lambda$ is simple, there is a second eigenvalue $\mu$ whose multiplicity is $t-1$. Since $S$ is real symmetric, the eigenspaces corresponding to $\lambda$ and $\mu$ are orthogonal and span $\mathbb R^t$. Hence the orthogonal projection onto the $\lambda$-eigenspace is $vv^{\top}$, whereas the projection onto its orthogonal complement is $I_t-vv^{\top}$. Consequently,
	$$
	S
	=
	\lambda vv^{\top}
	+
	\mu(I_t-vv^{\top}),
	$$
	or equivalently,
	$$
	S
	=
	\mu I_t+(\lambda-\mu)vv^{\top}.
	$$
	
	Since the diagonal of $S$ is zero, its $i$-th diagonal entry satisfies
	$
	0
	=
	S_{ii}
	=
	\mu+(\lambda-\mu)v_i^2
	$
	for every $i=1,\ldots,t$. Therefore, we have
	$
	v_i^2
	=
	-\tfrac{\mu}{\lambda-\mu}
	$
	is independent of $i$. Since $v$ is a unit vector,
	$
	\sum_{i=1}^{t}v_i^2=1,
	$
	and hence
	$
	v_i^2=\tfrac1t
	$
	for every $i$. The Perron vector is strictly positive, so the negative choice is excluded and
	$
	v_i=\tfrac1{\sqrt t}
	$
	for all $i$. Thus
	$
	v=\tfrac1{\sqrt t}\mathbf 1.
	$
	
	It is useful at this point to determine the relation between $\lambda$ and $\mu$. Since every diagonal entry of $S$ is zero,
	$
	\Tr(S)=0.
	$
	On the other hand, the eigenvalues of $S$ are $\lambda$ once and $\mu$ with multiplicity $t-1$. Hence
	$
	\lambda+(t-1)\mu=0,
	$
	which gives
	$
	\mu=-\tfrac{\lambda}{t-1}.
	$
	In particular, $\mu<0$.
	
	Substituting $v=\tfrac1{\sqrt t}\mathbf 1$ into the spectral decomposition gives
	$
	S	=	\mu I_t	+	\tfrac{\lambda-\mu}{t}J_t.
	$
	Using
	$
	\mu=-\tfrac{\lambda}{t-1},
	$
	we obtain
	$$
	\frac{\lambda-\mu}{t}
	=
	\frac1t
	\left(
	\lambda+\frac{\lambda}{t-1}
	\right)
	=
	\frac{\lambda}{t-1}.
	$$
	Therefore
	$
	S
	=
	\tfrac{\lambda}{t-1}(J_t-I_t).
	$
	Let
	$
	c=\tfrac{\lambda}{t-1}.
	$
	Since $\lambda>0$ and $t\geq3$, we have $c>0$. Hence every off-diagonal entry of $S$ equals $c$. By the definition of the symmetric quotient,
	$
	\sqrt{p_ip_j}\,\phi(d_i,d_j)=c
	$
	for every $i\neq j$, which is the required condition.
	
	It remains to justify the final assertion concerning regular complete multipartite graphs. Suppose
	$
	G=K_{\underbrace{p,\ldots,p}_{t}}.
	$
	Then
	$
	p_1=\cdots=p_t=p
	$
	and every vertex has the same degree
	$
	d=n-p=p(t-1).
	$
	Consequently, for every $i\neq j$,
	$$
	\sqrt{p_ip_j}\,\phi(d_i,d_j)
	=
	\sqrt{p^2}\,\phi(d,d)
	=
	p\phi(d,d).
	$$
	Thus the required constant condition holds with
	$
	c=p\phi(d,d)>0.
	$
	Hence every regular complete multipartite graph satisfies the criterion. The converse, however, is not valid for arbitrary degree-based weight functions, because equality of the quantities
	$
	\sqrt{p_ip_j}\,\phi(d_i,d_j)
	$
	does not in general force the part sizes $p_1,\ldots,p_t$ to be equal. 
\end{proof}

\begin{remark}\label{rem:correction_three_eigs}
	The ordinary quotient $M$ used in an equitable partition has entries $M_{ij}=p_j\phi(d_i,d_j)$ and is not symmetric unless the part sizes happen to agree in the appropriate way. Thus a spectral decomposition of $M$ as if it were symmetric is not justified. More importantly, equality of the off-diagonal entries of the symmetric quotient $S$, not equality of the part sizes, is the correct general condition. The next two propositions show that irregular examples occur even for familiar degree-based matrices.
\end{remark}

\begin{proposition}\label{prop:isi_9bb}
	Let $G=K_{a,b,b}$ with $a,b\ge 1$, $a+2b>3$, and let
	$
	\phi(x,y)=\frac{xy}{x+y}
	$
	be the ISI weight. Then the ISI matrix of $G$ has exactly three distinct eigenvalues if and only if
	$
	a=b$ or $a=9b.
	$
	In particular, for every $b\ge 1$, the irregular graph $K_{9b,b,b}$ satisfies
	$$
	\spec(ISI(K_{9b,b,b}))
	=\{0^{[11b-3]},10b^2,(-5b^2)^{[2]}\}.
	$$
\end{proposition}

\begin{proof}
	The degree in the part of size $a$ is $2b$, while the degree in either part of size $b$ is $a+b$. The two types of off-diagonal entries of the symmetric quotient are
	$$
	x=\sqrt{ab}\,\frac{2b(a+b)}{a+3b}
	\quad 
	\text{and}
	\quad
	y=b\,\frac{a+b}{2}.
	$$
	By Corollary~\ref{thm:tripartite_cubic_inertia}, exactly three distinct eigenvalues occur precisely when $x=y$. Cancelling the positive factor $b(a+b)$ gives
	$
	4\sqrt{ab}=a+3b.
	$
	With $z=\sqrt{a/b}$ this becomes
	$
	z^2-4z+3=0,
	$
	so $z=1$ or $z=3$. Hence $a=b$ or $a=9b$. If $a=9b$, then $x=y=5b^2$, and the displayed spectrum follows.
\end{proof}

\begin{proposition}\label{prop:m2_4bbb}
	Let $G=K_{a,b,b,b}$ with $a,b\ge 1$, and let $\phi(x,y)=xy$ be the second Zagreb weight. Then the corresponding weighted adjacency matrix has exactly three distinct eigenvalues if and only if
	$
	a=b$ or $a=4b.
	$
	For the irregular family $K_{4b,b,b,b}$,
	$$
	\spec(M2(K_{4b,b,b,b}))
	=\{0^{[7b-4]},108b^3,(-36b^3)^{[3]}\}.
	$$
\end{proposition}

\begin{proof}
	The degree in the part of size $a$ is $3b$, and the degree in each of the other three parts is $a+2b$. Let
	$
	x=3b(a+2b)\sqrt{ab},$ and $y=b(a+2b)^2.
	$
	Then the symmetric quotient is
	$$
	S=\begin{pmatrix}
		0&x&x&x\\
		x&0&y&y\\
		x&y&0&y\\
		x&y&y&0
	\end{pmatrix}.
	$$
	The value $-y$ is an eigenvalue of multiplicity $2$, coming from vectors supported on the last three coordinates whose sum is zero. On the remaining two-dimensional invariant subspace the eigenvalues are
	$
	y\pm\sqrt{y^2+3x^2}.
	$
	These three reduced values collapse to two distinct eigenvalues exactly when
	$
	y-\sqrt{y^2+3x^2}=-y,
	$
	which is equivalent to $x=y$. After cancellation, we get
	$
	3\sqrt{ab}=a+2b.
	$
	Putting $z=\sqrt{a/b}$ gives $z^2-3z+2=0$, hence $a=b$ or $a=4b$. For $a=4b$, we obtain $x=y=36b^3$, and therefore $S=36b^3(J_4-I_4)$. Thus the results follows.
\end{proof}

The regular case remains particularly simple and is worth recording in a form that also gives the energy.

\begin{corollary}\label{cor:regular_multipartite_spectrum_updated}
	Let $G=K_{p,\ldots,p}$ be a complete $t$-partite graph, where $t\ge 2$, $p\ge 1$, $n=pt$, $d=p(t-1)$, and $w=\phi(d,d)$. Then
	$$
	\spec(\dot A(G))
	=\{0^{[t(p-1)]},((t-1)pw)^{[1]},(-pw)^{[t-1]}\}.
	$$
	If $w>0$, then
	$
	\rho(\dot A(G))=(t-1)pw
	$
	and
	$
	\dot E_\phi(G)=2(t-1)pw.
	$
	For $p\ge 2$, the matrix has exactly three distinct eigenvalues.
\end{corollary}

\begin{proof}
	Here every entry of the symmetric quotient off the diagonal is
	$
	\sqrt{p^2}\phi(d,d)=pw.
	$
	Thus
	$
	S=pw(J_t-I_t),
	$
	whose eigenvalues are $(t-1)pw$ once and $-pw$ with multiplicity $t-1$. The remaining $t(p-1)$ eigenvalues are zero. The formulas for the spectral radius and energy follow immediately.
\end{proof}

\begin{corollary}\label{cor:integrality_regular_multipartite_updated}
	Under the hypotheses of Corollary~\ref{cor:regular_multipartite_spectrum_updated}, the weighted spectrum is integral if and only if
	$
	p\phi(d,d)\in\mathbb Z.
	$
\end{corollary}

\begin{proof}
	The nonzero eigenvalues are $(t-1)p\phi(d,d)$ and $-p\phi(d,d)$. Hence integrality of the full spectrum is equivalent to integrality of the single quantity $p\phi(d,d)$.
\end{proof}

\begin{example}\label{ex:integrality_correction}
	The factor $p$ in Corollary~\ref{cor:integrality_regular_multipartite_updated} cannot be omitted. On $K_{2,2,2}$, take a degree weight satisfying $\phi(4,4)=\tfrac12$. Then
	$$
	\spec(\dot A(K_{2,2,2}))=\{0^{[3]},2,(-1)^{[2]}\},
	$$
	which is integral although $\phi(4,4)=\tfrac12\notin\mathbb Z$.
\end{example}

The following consequence is related to the spectra of the regular complete multipartite graph $K_{p,\ldots,p}.$
\begin{corollary}\label{cor:integrality_examples_updated}
	Let $G=K_{p,\ldots,p}$ with $t\ge 2$, $p\ge 2$, and $d=p(t-1)$. Then the following statements hold.
	\begin{enumerate}
		\item For the adjacency, arithmetic--geometric, geometric--arithmetic, first Zagreb, and second Zagreb weights, the weighted spectrum is always integral.
		\item For the ISI weight $\phi(x,y)=\tfrac{xy}{x+y}$, the spectrum is integral if and only if $d$ is even, equivalently, if and only if $p$ is even or $t$ is odd.
		\item For the Randic weight $\phi(x,y)=\tfrac{1}{\sqrt{xy}}$, the spectrum is integral if and only if $t=2$.
		\item For the atom--bond connectivity weight $\phi(x,y)=\sqrt{\tfrac{x+y-2}{xy}}$, the spectrum is integral if and only if $t=2$ and
		$
		p=2m^2+1
		$
		for some positive integer $m$.
		\item For the Sombor weight $\phi(x,y)=\sqrt{x^2+y^2}$ and the modified Sombor weight $\phi(x,y)=\tfrac{1}{\sqrt{x^2+y^2}}$, the weighted spectrum is never integral.
	\end{enumerate}
\end{corollary}

\begin{proof}
	By Corollary~\ref{cor:integrality_regular_multipartite_updated}, it is enough to test $p\phi(d,d)$.
	 For adjacency, AG, and GA, one has $\phi(d,d)=1$. For the first and second Zagreb weights one has $\phi(d,d)=2d$ and $\phi(d,d)=d^2$, respectively. These cases are therefore always integral.
	
	For ISI,
	$$
	p\phi(d,d)=\frac{pd}{2}=\frac{p^2(t-1)}{2}.
	$$
	This is integral exactly when $p^2(t-1)$ is even. Since $p$ and $p^2$ have the same parity, this is equivalent to $d=p(t-1)$ being even.
	
	For the Randic weight,
	$$
	p\phi(d,d)=\frac{p}{d}=\frac{1}{t-1},
	$$
	which is an integer precisely when $t=2$.
	
	For the ABC weight,
	$$
	p\phi(d,d)=\frac{\sqrt{2(d-1)}}{t-1}.
	$$
	Suppose this equals a positive integer $k$ and put $s=t-1$. Then
	$
	2(ps-1)=k^2s^2.
	$
	It follows that $s$ divides $2$, so $s=1$ or $s=2$. If $s=2$, the left side is congruent to $2$ modulo $4$, whereas the right side is divisible by $4$, a contradiction. Hence $s=1$, and
	$
	2(p-1)=k^2.
	$
	Thus $k=2m$ and $p=2m^2+1$. The converse is immediate.
	
	Finally,
	$
	p\sqrt{d^2+d^2}=pd\sqrt2
	$
	and
	$
	\tfrac{p}{\sqrt{d^2+d^2}}=\tfrac{1}{(t-1)\sqrt2},
	$
	so both Sombor quantities are irrational.
\end{proof}

The regular formula also produces several compact energy identities that are useful for checking computations.

\begin{table}[H]
	\centering
	\caption{Closed forms for regular complete multipartite graphs $K_{p,\ldots,p}$, where $d=p(t-1)$.}
	\label{tab:regular_energy_identities} 
		\begin{tabular}{|c|c|c|c|c|}
			\hline
			Matrix & $\phi(d,d)$ & $p\phi(d,d)$ & $\rho(\dot A(G))$ & $\dot E_\phi(G)$\\
			\hline
			$A$ & $1$ & $p$ & $d$ & $2d$\\
			$AG$ & $1$ & $p$ & $d$ & $2d$\\
			$GA$ & $1$ & $p$ & $d$ & $2d$\\
			$ISI$ & $\tfrac d2$ & $\tfrac{pd}{2}$ & $\tfrac{d^2}{2}$ & $d^2$\\
			$M1$ & $2d$ & $2pd$ & $2d^2$ & $4d^2$\\
			$ABC$ & $\tfrac{\sqrt{2(d-1)}}{d}$ & $\tfrac{\sqrt{2(d-1)}}{t-1}$ & $\sqrt{2(d-1)}$ & $2\sqrt{2(d-1)}$\\
			$R$ & $\tfrac1d$ & $\tfrac1{t-1}$ & $1$ & $2$\\
			$M2$ & $d^2$ & $pd^2$ & $d^3$ & $2d^3$\\
			$S$ & $d\sqrt2$ & $pd\sqrt2$ & $d^2\sqrt2$ & $2d^2\sqrt2$\\
			$MS$ & $\tfrac1{d\sqrt2}$ & $\tfrac1{(t-1)\sqrt2}$ & $\tfrac1{\sqrt2}$ & $\sqrt2$\\
			\hline
	\end{tabular}
\end{table}

Table \ref{tab:regular_energy_identities} gives the closed form of energy and the spectral radius with weights for regular complete multipartite graphs $K_{p,\ldots,p}.$

\begin{theorem}\label{thm:multipartite_rho_bounds}
	Let $G=K_{p_1,\ldots,p_t}$ and suppose $\phi(d_i,d_j)>0$ for $i\ne j$. Define the weighted row sum of a vertex in the $i$-th part by
	$
	r_i=\sum_{j\ne i}p_j\phi(d_i,d_j).
	$
	Then
	$$
	\min_{1\le i\le t}r_i\le \rho(\dot A(G))\le \max_{1\le i\le t}r_i.
	$$
	In addition,
	$$
	\rho(\dot A(G))\ge
	\frac{2}{n}\sum_{1\le i<j\le t}p_ip_j\phi(d_i,d_j).
	$$
	Equality in the last bound holds if and only if $r_1=\cdots=r_t$.
\end{theorem}

\begin{proof}
	The ordinary quotient $M$ is nonnegative and irreducible, and its row sums are $r_1,\ldots,r_t$. Since $M$ is similar to $S$, it has spectral radius $\rho(\dot A(G))$. The first pair of inequalities is therefore the standard row-sum bound for an irreducible nonnegative matrix.
	
	For the second bound, use the all-ones vector in the Rayleigh quotient of $\dot A(G)$. Since every edge between $V_i$ and $V_j$ has weight $\phi(d_i,d_j)$ and there are $p_ip_j$ such edges,
	$$
	\rho(\dot A(G))\ge
	\frac{\mathbf 1^{\top}\dot A(G)\mathbf 1}{\mathbf 1^{\top}\mathbf 1}
	=\frac{2}{n}\sum_{i<j}p_ip_j\phi(d_i,d_j).
	$$
	For a connected positive weighted graph, equality in the Rayleigh quotient occurs exactly when $\mathbf 1$ is a Perron eigenvector, which is equivalent to all weighted row sums being equal.
\end{proof}

\begin{example}\label{ex:rho_bounds_k311}
	For the ordinary adjacency matrix of $K_{3,1,1}$, the three partwise row sums are $2,4,4$. The graph has seven edges and order five, while Corollary~\ref{thm:tripartite_cubic_inertia} give $\rho=3$. Thus
	$$
	2\le \frac{14}{5}=2.8\le 3\le 4,
	$$
	which illustrates both bounds in Theorem~\ref{thm:multipartite_rho_bounds}.
\end{example}

\begin{table}[H]
	\centering
	\caption{Numerical and exact checks for the structural results of the section.}
	\label{tab:structural_numerics}
	\resizebox{\textwidth}{!}{%
		\begin{tabular}{|c|c|c|c|c|c|c|}
			\hline
			Graph & Weight & $\spec(S)$ & $\operatorname{mult}(0)$ & Distinct eigenvalues & Energy & Integral?\\
			\hline
			$K_{4,3,2,1}$ & Example~\ref{ex:extra_zero_multipartite} & $\{5,0,-1,-4\}$ & $7$ & $4$ & $10$ & Yes\\
			$K_{3,1,1}$ & $A$ & $\{3,-1,-2\}$ & $2$ & $4$ & $6$ & Yes\\
			$K_{2,2,2}$ & $ISI$ & $\{8,-4,-4\}$ & $3$ & $3$ & $16$ & Yes\\
			$K_{9,1,1}$ & $ISI$ & $\{10,-5,-5\}$ & $8$ & $3$ & $20$ & Yes\\
			$K_{4,1,1,1}$ & $M2$ & $\{108,-36,-36,-36\}$ & $3$ & $3$ & $216$ & Yes\\
			$K_{3,3}$ & $ABC$ & $\{2,-2\}$ & $4$ & $3$ & $4$ & Yes\\
			\hline
	\end{tabular}}
\end{table}

Table \ref{tab:structural_numerics} gives the numerical and exact checks  for spectra, energy and integral. In each row, the displayed zero multiplicity is the multiplicity in the full weighted adjacency matrix.

\begin{table}[H]
	\centering
	\caption{Numerical verification of the regular formulas.}
	\label{tab:regular_numerical_validation}
	\resizebox{\textwidth}{!}{%
		\begin{tabular}{|c|rrrr|rrrr|}
			\hline
			& \multicolumn{4}{c|}{$K_{2,2,2}$} & \multicolumn{4}{c|}{$K_{3,3,3,3}$}\\
			\cline{2-9}
			Matrix & $a$ & $\rho$ & Energy & Int. & $a$ & $\rho$ & Energy & Int.\\
			\hline
			$A$   & $2$ & $4$ & $8$ & Yes & $3$ & $9$ & $18$ & Yes\\
			$AG$  & $2$ & $4$ & $8$ & Yes & $3$ & $9$ & $18$ & Yes\\
			$GA$  & $2$ & $4$ & $8$ & Yes & $3$ & $9$ & $18$ & Yes\\
			$ISI$ & $4$ & $8$ & $16$ & Yes & $13.5$ & $40.5$ & $81$ & No\\
			$M1$  & $16$ & $32$ & $64$ & Yes & $54$ & $162$ & $324$ & Yes\\
			$ABC$ & $1.224745$ & $2.449490$ & $4.898979$ & No & $1.333333$ & $4$ & $8$ & No\\
			$R$   & $0.5$ & $1$ & $2$ & No & $0.333333$ & $1$ & $2$ & No\\
			$M2$  & $32$ & $64$ & $128$ & Yes & $243$ & $729$ & $1458$ & Yes\\
			$S$   & $11.313708$ & $22.627417$ & $45.254834$ & No & $38.183766$ & $114.551299$ & $229.102597$ & No\\
			$MS$  & $0.353553$ & $0.707107$ & $1.414214$ & No & $0.235702$ & $0.707107$ & $1.414214$ & No\\
			\hline
	\end{tabular}}
\end{table}

Table \ref{tab:regular_numerical_validation} verification for regular multipartite graphs. For $K_{2,2,2}$, $d=4$ and the zero multiplicity is $3$; for $K_{3,3,3,3}$, $d=9$ and the zero multiplicity is $8$. The column $a$ denotes $p\phi(d,d)$, so the nonzero spectrum is $\{(t-1)a,(-a)^{[t-1]}\}$.

The calculations in Tables~\ref{tab:structural_numerics} and \ref{tab:regular_numerical_validation} also separate the three issues that are easy to conflate: the forced zero eigenvalues coming from repeated rows inside the partite sets, possible additional zeros of the reduced matrix, and arithmetic integrality of the nonzero reduced eigenvalues. The symmetric quotient $S$ keeps these points visible and avoids treating a nonsymmetric equitable quotient as though it admitted an orthogonal spectral decomposition.

\section{Weighted adjacency energy of some special families of graphs}\label{section 3}

Throughout this section we assume that $\phi(x,y)=\phi(y,x)>0$ on the degree pairs that occur on edges. This is the natural setting for the degree--based matrices considered in Table~\ref{table matrices}, and it allows us to use Perron--Frobenius theory without adding sign qualifications at every step. The calculations below also correct several statements in the original version of this section. 

\medskip

Let $e=v_1v_2$ be an edge of $K_n$, where $n\geq 3$, and let $G=K_n-e$. We use the abbreviations
$
b_n=\phi(n-1,n-1),$ and $a_n=\phi(n-2,n-1).
$
Thus $b_n$ is the weight of every edge of $K_n$, while $a_n$ is the weight of an edge of $G$ incident with one of the two vertices $v_1,v_2$ whose degree drops from $n-1$ to $n-2$.

\begin{proposition}\label{thm:Kn-minus-edge-exact}
	For $n\geq 3$, the weighted adjacency spectrum of $K_n$ is
	$ \Big\{(n-1)b_n,\;(-b_n)^{[n-1]}\Big\}, $
	and
	$
	\dot E_\phi(K_n)=2(n-1)b_n.
	$
	Moreover, the spectrum of $K_n-e$ is
	$$
	\spec(\dot A(K_n-e))
	=
	\Big\{0,\;(-b_n)^{[n-3]},\;\lambda_+,\;\lambda_-\Big\},
	$$
	where
	$
	\lambda_{\pm}
	=
	\tfrac{(n-3)b_n\pm\Delta_n}{2},
	$ and $
	\Delta_n=
	\sqrt{(n-3)^2b_n^2+8(n-2)a_n^2}.
	$
	Consequently,
	$
	\rho(\dot A(K_n-e))=\lambda_+
	$
	and
	$
	\dot E_\phi(K_n-e)=(n-3)b_n+\Delta_n.
	$
\end{proposition}

\begin{proof}
	Since $K_n$ is $(n-1)$-regular, every edge has weight $b_n$, and hence
	$
	\dot A(K_n)=b_n A(K_n)=b_n(J_n-I_n).
	$
	The first spectrum and energy formula follow immediately.
	
	For $G=K_n-e$, order the vertices as $v_1,v_2,v_3,\ldots,v_n$. Then
	$$
	\dot A(G)=
	\begin{pmatrix}
		0_2 & a_nJ_{2\times(n-2)}\\
		a_nJ_{(n-2)\times2} & b_n(J_{n-2}-I_{n-2})
	\end{pmatrix}.
	$$
	The vector $e_1-e_2$ is a $0$-eigenvector. The subspace formed by vectors supported on $\{v_3,\ldots,v_n\}$ whose coordinates sum to $0$ has dimension $n-3$ and carries the eigenvalue $-b_n$. The remaining two eigenvalues are obtained from the equitable quotient
	$$
	Q_n=
	\begin{pmatrix}
		0&(n-2)a_n\\
		2a_n&(n-3)b_n
	\end{pmatrix}.
	$$
	Its characteristic polynomial is
	$
	x^2-(n-3)b_nx-2(n-2)a_n^2,
	$
	which gives $\lambda_\pm$. Since the constant term is negative, $\lambda_+>0>\lambda_-$. The matrix is nonnegative and irreducible, so $\lambda_+$ is its spectral radius. Finally,
	$$
	\dot E_\phi(G)
	=(n-3)b_n+|\lambda_+|+|\lambda_-|
	=(n-3)b_n+\Delta_n.
	$$
\end{proof}

The next statement gives the edge--deletion criterion for the weighted adjacency matrix.
\begin{theorem}\label{thm:Kn-edge-trichotomy}
	Let
	$
	R_n=\tfrac{b_n}{a_n}
	$ and $
	\tau_n=\sqrt{\tfrac{n-2}{n-1}}.
	$
	Then
	$$
	\dot E_\phi(K_n-e)-\dot E_\phi(K_n)
	=
	2\Big(\rho(\dot A(K_n-e))-\rho(\dot A(K_n))\Big).
	$$
	More precisely,
	\begin{enumerate}
		\item if $R_n<\tau_n$, then both the spectral radius and the energy strictly increase after deleting $e$;
		\item if $R_n=\tau_n$, then both quantities remain unchanged;
		\item if $R_n>\tau_n$, then both the spectral radius and the energy strictly decrease after deleting $e$.
	\end{enumerate}
\end{theorem}

\begin{proof}
	By Proposition~\ref{thm:Kn-minus-edge-exact}, we have
	$$
	\rho(\dot A(K_n-e))-\rho(\dot A(K_n))
	=
	\frac{\Delta_n-(n+1)b_n}{2},
	$$
	while
	$$
	\dot E_\phi(K_n-e)-\dot E_\phi(K_n)
	=\Delta_n-(n+1)b_n.
	$$
	This proves the first identity. Since all quantities are positive, so
	$
	\Delta_n>(n+1)b_n
	$
	is equivalent to
	$
	(n-2)a_n^2>(n-1)b_n^2,
	$
	that is,
	$$
	\frac{b_n}{a_n}<\sqrt{\frac{n-2}{n-1}}.
	$$
	The equality and reverse inequality cases follow in the same way.
\end{proof}

The following consequence follows form above result and it corrects Theorem 2  of Bilal and Munir \cite{bilalpak}.
\begin{corollary}\label{cor:ISI-Kn-deletion}
	For the inverse sum indeg weight
	$
	\phi(x,y)=\tfrac{xy}{x+y},
	$
	the spectral radius and the $ISI$ energy of $K_n$ strictly decrease after deleting any edge, for every $n\geq 3$.
\end{corollary}

\begin{proof}
	As
	$
	b_n=\tfrac{n-1}{2},
	$  $
	a_n=\tfrac{(n-2)(n-1)}{2n-3},
	$
	and hence
	$$
	R_n=\frac{2n-3}{2n-4}>1>\sqrt{\frac{n-2}{n-1}}.
	$$
	The conclusion follows from Theorem~\ref{thm:Kn-edge-trichotomy}.
\end{proof}

\begin{example}\label{ex:ISI-K23-corrected}
	For $n=23$, the $ISI$ spectrum of $K_{23}$ is
	$
	\Big\{242,\;(-11)^{[22]}\Big\},
	$
	so $\dot E_{ISI}(K_{23})=484$. For $K_{23}-e$, we have
	$$
	\spec(ISI(K_{23}-e))
	\approx
	\Big\{240.185930,\;0,\;(-11)^{[20]},\;-20.185930\Big\},
	$$
	and therefore
	$$
	\dot E_{ISI}(K_{23}-e)\approx480.371860<484.
	$$
	This example agrees with the corrected direction in Theorem~\ref{thm:Kn-edge-trichotomy}.
\end{example}

The exact threshold also gives a clean classification for the general Randi\`c family.

\begin{corollary}\label{cor:general-Randic-Kn}
	Let $\phi_\alpha(x,y)=(xy)^\alpha$, where $\alpha\in\mathbb R$. For every $n\geq3$, deletion of one edge from $K_n$ has the following effect on both weighted spectral radius and weighted energy:
	\begin{enumerate}
		\item they increase if $\alpha<-	\tfrac12$;
		\item they remain unchanged if $\alpha=-	\tfrac12$;
		\item they decrease if $\alpha>-	\tfrac12$.
	\end{enumerate}
\end{corollary}

\begin{proof}
	For this family,
	$$
	R_n
	=
	\frac{((n-1)^2)^\alpha}{((n-2)(n-1))^\alpha}
	=
	\left(\frac{n-1}{n-2}\right)^\alpha,
	$$
	whereas
	$$
	\tau_n
	=
	\left(\frac{n-1}{n-2}\right)^{-1/2}.
	$$
	Since $\tfrac{n-1}{n-2}>1$, the comparison between $R_n$ and $\tau_n$ is exactly the comparison between $\alpha$ and $-	\tfrac12$.
\end{proof}

\begin{remark}
	The case $\alpha=-\tfrac12$ is the Randi\`c matrix. Thus
	$$
	\rho(R(K_n))=\rho(R(K_n-e))=1
	$$
	and
	$$
	\dot E_R(K_n)=\dot E_R(K_n-e)=2.
	$$
	This equality case shows that a blanket conjecture asserting strict energy decrease for ``almost all'' degree weights is unnecessary: Theorem~\ref{thm:Kn-edge-trichotomy} gives the exact answer, including increase, equality, and decrease.
\end{remark}

\begin{table}[H]
	\centering 
	\setlength{\tabcolsep}{3.5pt}
	\begin{tabular}{|c|c|c|c|c|c|c|}
		\hline
		Matrix & $R_{25}$ & $\rho(K_{25})$ & $\rho(K_{25}-e)$ & $\dot E(K_{25})$ & $\dot E(K_{25}-e)$ & Effect\\
		\hline
		$ISI$ & $1.021739$ & $288.000000$ & $286.172392$ & $576.000000$ & $572.344783$ & decrease\\
		\hline
		$ABC$ & $0.989762$ & $6.782330$ & $6.770971$ & $13.564660$ & $13.541941$ & decrease\\
		\hline
		$A$ & $1.000000$ & $24.000000$ & $23.922848$ & $48.000000$ & $47.845696$ & decrease\\
		\hline
		$AG$ & $0.999774$ & $24.000000$ & $23.923654$ & $48.000000$ & $47.847308$ & decrease\\
		\hline
		$GA$ & $1.000226$ & $24.000000$ & $23.922042$ & $48.000000$ & $47.844085$ & decrease\\
		\hline
		$M_1$ & $1.021277$ & $1152.000000$ & $1144.764136$ & $2304.000000$ & $2289.528273$ & decrease\\
		\hline
		$R$ & $0.978945$ & $1.000000$ & $1.000000$ & $2.000000$ & $2.000000$ & unchanged\\
		\hline
		$M_2$ & $1.043478$ & $13824.000000$ & $13695.434761$ & $27648.000000$ & $27390.869522$ & decrease\\
		\hline
		$S$ & $1.021046$ & $814.587012$ & $809.496852$ & $1629.174024$ & $1618.993704$ & decrease\\
		\hline
		$MS$ & $0.979388$ & $0.707107$ & $0.707058$ & $1.414214$ & $1.414115$ & decrease\\
		\hline
		$(xy)^{-1}$ & $0.958333$ & $0.041667$ & $0.041806$ & $0.083333$ & $0.083611$ & increase\\
		\hline
	\end{tabular}
	\caption{Correct numerical comparison for $K_{25}$ and $K_{25}-e$.}
	\label{tab:Kn-corrected-numerics}
\end{table}
Table \ref{tab:Kn-corrected-numerics} shows the increase/decrease numerical check for $K_{25}$ and $K_{25}-e$. Here $\tau_{25}=\sqrt{23/24}\approx0.978945$. The last row illustrates the increasing regime, while the Randi\`c row gives exact equality.
Thus, the $ISI$ energy $\dot{E}_{ISI}$ of $K_{n}$ does not increase upon edge deletion, which contradicts Theorem 2 presented by Bilal and Munir \cite{bilalpak}.
\subsection{The $ISI$ matrix of a regular complete multipartite graph after one edge deletion}

Let
$K_{p}^{(t)}=K_{\underbrace{p,p,\ldots,p}_{t}},~ t\geq3,$ and $ p\geq2$ be the complete multipartite graph.
Every vertex of $K_p^{(t)}$ has degree
$
r=p(t-1),
$ and hence
$$
ISI(K_p^{(t)})=\frac{r}{2}A(K_p^{(t)}),
$$
and therefore
$$
\dot E_{ISI}(K_p^{(t)})=r^2=p^2(t-1)^2.
$$
Notice, however, that $K_p^{(t)}-e$ is not regular. Thus the identity $ISI(G)=\tfrac r2A(G)$ cannot be applied to the graph after deletion. This is precisely why an adjacency-energy theorem for a regular graph does not by itself settle the corresponding $ISI$ problem after an edge is removed.

Let $e=uv$ join the first and second partite sets and put
$
G_{t,p}=K_p^{(t)}-e.
$
The two exceptional vertices have degree $r-1$, all other vertices have degree $r$, and the two edge weights that occur in $G_{t,p}$ are
$a=\tfrac{r(r-1)}{2r-1},$ and $b=\tfrac r2.$

The following result gives the $ISI$ spectrum of multipartite graph upon edge deletion.
\begin{theorem}\label{thm:ISI-general-multipartite-minus-edge}
	For $t\geq3$ and $p\geq2$, the $ISI$ spectrum of $G_{t,p}$ is
	$$
	\spec(ISI(G_{t,p}))
	=
	\Big\{0^{[pt-t-2]},\;(-bp)^{[t-3]},\;\lambda_+,\lambda_-,\;\mu_1,\mu_2,\mu_3\Big\},
	$$
	where
	$$
	\lambda_{\pm}
	=
	\frac{-b(p-1)\pm\sqrt{b^2(p-1)^2+4a^2(p-1)}}{2},
	$$
	and $\mu_1,\mu_2,\mu_3$ are the eigenvalues of
	$$
	M_{t,p}=
	\begin{pmatrix}
		0&a(p-1)&ap(t-2)\\
		a&b(p-1)&bp(t-2)\\
		2a&2b(p-1)&bp(t-3)
	\end{pmatrix}.
	$$
	In particular,
	$$
	\det(M_{t,p})
	=
	\frac{p^4(p-1)(t-1)^4(pt-p-1)^2}{2(2pt-2p-1)^2}>0,
	$$
	so no additional zero eigenvalue is hidden in the $3\times3$ block.
\end{theorem}

\begin{proof}
	Let the deleted edge be $u\in V_1$ and $v\in V_2$. Now, split $V_1$ into $\{u\}$ and $V_1\setminus\{u\}$, split $V_2$ similarly, and keep each of $V_3,\ldots,V_t$ as a cell. The vertices inside every nonsingleton cell have identical weighted rows outside that cell. Hence the usual zero-sum subspaces inside the cells contribute
	$$
	2(p-2)+(t-2)(p-1)=pt-t-2
	$$
	zero eigenvalues.
	
	On the quotient space, first use the involution interchanging the first two partite sets. Its antisymmetric subspace has basis corresponding to
	$
	\{u\}-\{v\},
	$ and $
	(V_1\setminus\{u\})-(V_2\setminus\{v\}),
	$
	and the induced matrix is
	$$
	A_-=
	\begin{pmatrix}
		0&-a(p-1)\\
		-a&-b(p-1)
	\end{pmatrix}.
	$$
	The eigenvalues of above matrix gives $\lambda_\pm$.
	
	Next take vectors that are constant on each of $V_3,\ldots,V_t$, vanish on the first four cells, and whose constants sum to zero. This space has dimension $t-3$ and carries the eigenvalue $-bp$.
	
	The remaining three-dimensional symmetric space consists of vectors having a common value on $\{u,v\}$, a common value on $(V_1\setminus\{u\})\cup(V_2\setminus\{v\})$, and a common value on every one of $V_3,\ldots,V_t$. Direct row-sum calculation gives $M_{t,p}$. The determinant displayed in the statement follows by simplification. The listed multiplicities add up to $pt$, completing the spectrum.
\end{proof}

\begin{figure}[H]
	\centering
	\begin{tikzpicture}[scale=1.05,every node/.style={circle,draw,minimum size=5.5mm,inner sep=1pt}]
		\node (a1) at (-3,1) {$u$};\node (a2) at (-3,-1) {$a$};
		\node (b1) at (0,.65) {$v$};\node (b2) at (0,-0.65) {$b$};
		\node (c1) at (3,1) {$c$};\node (c2) at (3,-1) {$d$};
		\foreach \x in {a1,a2}{\foreach \y in {c1,c2}{\draw (\x)--(\y);}}
		\foreach \x in {b1,b2}{\foreach \y in {c1,c2}{\draw (\x)--(\y);}}
		\draw (a1)--(b2);\draw (a2)--(b1);\draw (a2)--(b2);
		\draw[dashed] (a1)--node[above,draw=none,rectangle, yshift=-4.5mm] {$e$} (b1);
		\node[draw=none,rectangle] at (-3,-1.7) {$V_1$};
		\node[draw=none,rectangle] at (0,-1.7) {$V_2$};
		\node[draw=none,rectangle] at (3,-1.7) {$V_3$};
	\end{tikzpicture}
	\caption{The smallest exceptional case $K_{2,2,2}-e$. The dashed edge $e$ is removed. }
	\label{fig:K222-minus-edge}
\end{figure}

For $t=3$, Theorem~\ref{thm:ISI-general-multipartite-minus-edge} gives a particularly transparent correction of the earlier five-dimensional quotient calculation. Also, note that graph in Figure \ref{fig:K222-minus-edge} is the unique exception to the general $ISI$ energy-increase statement below.
\begin{corollary}\label{cor:ISI-tripartite-corrected-spectrum}
	Let $G=K_{p,p,p}-e$ with $p\geq2$, and let
	$
	a=\tfrac{2p(2p-1)}{4p-1}.
	$
	Then
	$$
	\spec(ISI(G))
	=
	\Big\{0^{[3p-5]},\lambda_+,\lambda_-,\mu_1,\mu_2,\mu_3\Big\},
	$$
	where
	$$
	\lambda_{\pm}
	=
	\frac{-p(p-1)\pm\sqrt{p^2(p-1)^2+4a^2(p-1)}}{2},
	$$
	and $\mu_1,\mu_2,\mu_3$ are the zeros of the monic cubic
	$$
	\begin{aligned}
		q_p(x)={}&x^3-p(p-1)x^2-\frac{2p^2(16p^4-23p^2+13p-2)}{(4p-1)^2}x
		-\frac{8p^4(p-1)(2p-1)^2}{(4p-1)^2}.
	\end{aligned}
	$$ 
\end{corollary}

\begin{proposition}\label{prop:ISI-tripartite-energy}
	Let $\mu_1$ denote the largest zero of $q_p$. Then $q_p$ has one positive and two negative zeros and
	$$
	\dot E_{ISI}(K_{p,p,p}-e)=2(\lambda_++\mu_1).
	$$
\end{proposition}

\begin{proof}
	The coefficient of $x$ in $q_p$ is negative , since
	$
	16p^4-23p^2+13p-2>0
	$
	for $p\geq2$. Hence the sum $\mu_1\mu_2+\mu_1\mu_3+\mu_2\mu_3$ is negative, so the three real eigenvalues of the symmetrizable quotient cannot all be positive. On the other hand,
	$
	\det(M_{3,p})>0,
	$
	and the Perron eigenvalue is positive. Thus $M_{3,p}$ has exactly one positive and two negative eigenvalues. Since $\lambda_+>0>\lambda_-$, the full matrix has exactly two positive eigenvalues, namely $\lambda_+$ and $\mu_1$. Its trace is $0$, so its energy is twice the sum of its positive eigenvalues.
\end{proof}

Now, we  have the multipartite edge-deletion assertion. 
\begin{theorem}\label{thm:ISI-multipartite-energy-change-corrected}
	Let $t\geq3$ and $p\geq2$. Then
	$$
	\dot E_{ISI}(K_p^{(t)}-e)>\dot E_{ISI}(K_p^{(t)})
	$$
	for every pair $(t,p)$ except $(t,p)=(3,2)$. For the exceptional graph (see Figure \ref{fig:K222-minus-edge}),
	$$
	\dot E_{ISI}(K_{2,2,2}-e)\approx15.816551<16
	=
	\dot E_{ISI}(K_{2,2,2}).
	$$
	Equivalently, the classification is
	\begin{enumerate}
		\item if $t=3$, the energy decreases for $p=2$ and increases for every $p\geq3$;
		\item if $t\geq4$, the energy increases for every $p\geq2$.
	\end{enumerate}
\end{theorem}

\begin{proof}
	Let $r=p(t-1)$ and retain the notation $a,b$ above. The original regular graph satisfies
	$
	\dot E_{ISI}(K_p^{(t)})=r^2.
	$
	The positive number $\lambda_+$ in Theorem~\ref{thm:ISI-general-multipartite-minus-edge} is an eigenvalue of $ISI(G_{t,p})$ with a sign-changing eigenvector, and hence it is distinct from the Perron root $\rho$. Since the trace is $0$, so we have
	$$
	\dot E_{ISI}(G_{t,p})
	=2\sum_{\lambda_i>0}\lambda_i
	\geq2(\lambda_++\rho).
	$$
	The total weighted edge sum of the original graph is $tp\,r^2/4$. By deleting $e$ removes one edge of weight $b$ and changes the weights on exactly $2(r-1)$ surviving edges from $b$ to $a$. Since
	$
	2(r-1)(b-a)=a,
	$
	the weighted edge sum of $G_{t,p}$ is
	$$
	\frac{tp\,r^2}{4}-a-b.
	$$
	Therefore, the Rayleigh quotient of the all-ones vector gives
	$$
	\rho\geq\frac{r^2}{2}-\frac{2(a+b)}{tp}.
	$$
	Thus, it is enough to prove that
	$$
	\lambda_+>\frac{2(a+b)}{tp}.
	$$
	Now $\lambda_+$ is the positive zero of
	$$
	f(x)=x^2+b(p-1)x-a^2(p-1).
	$$
	Let
	$
	c=\tfrac{2(a+b)}{tp}.
	$
	A direct simplification gives
	$$
	f(c)
	=-\frac{(t-1)^2P(t,p)}{2t^2(2pt-2p-1)^2},
	$$
	where
	$$
	\begin{aligned}
		P(t,p)={}&2p^5t^4-4p^5t^3+2p^5t^2-2p^4t^4-8p^4t^3+18p^4t^2-8p^4t+12p^3t^3-8p^3t^2-2p^3t\\
		&-44p^2t^2+71p^2t-32p^2+51pt-48p-18.
	\end{aligned}
	$$
	For $t=3$ and $p=3+v$, $v\geq0$, this becomes
	$$
	P(3,3+v)
	=72v^5+840v^4+3846v^3+8479v^2+8697v+3060>0.
	$$
	For $t=4+u$ and $p=2+v$, where $u,v\geq0$, write
	$$
	P(4+u,2+v)=A_4(v)u^4+A_3(v)u^3+A_2(v)u^2+A_1(v)u+A_0(v),
	$$
	with
	$$
	\begin{aligned}
		A_4(v)&=2v^5+18v^4+64v^3+112v^2+96v+32,\\
		A_3(v)&=28v^5+240v^4+812v^3+1352v^2+1104v+352,\\
		A_2(v)&=146v^5+1190v^4+3816v^3+5972v^2+4496v+1264,\\
		A_1(v)&=336v^5+2600v^4+7870v^3+11419v^2+7607v+1650,\\
		A_0(v)&=288v^5+2112v^4+6008v^3+7948v^2+4396v+470.
	\end{aligned}
	$$
	Every coefficient is positive. Hence $P(t,p)>0$ for $t\geq4$, $p\geq2$, and also for $t=3$, $p\geq3$. Thus $f(c)<0$ in precisely these ranges. Since $f$ opens upward and has one positive zero, $c<\lambda_+$, which yields
	$
	\dot E_{ISI}(G_{t,p})>r^2.
	$
	For $(t,p)=(3,2)$, direct evaluation of the five nonzero quotient eigenvalues gives
	$
	6.923641,~0.984635,~-0.829363,~-2.984635,$ and $-4.094278,
	$
	so the energy is approximately $15.816551<16$.
\end{proof}

Bilal and Munir \cite{bilalpak} also proved that the $ISI$ energy of a regular tripartite graph $K_{p,p,p}$ decreases upon edge deletion (see Theorems 6 and 7 of \cite{bilalpak}). However, for $t=3,$ Theorem \ref{thm:ISI-multipartite-energy-change-corrected} discards their results.

\begin{table}[H]
	\centering
	\small
	\begin{tabular}{|c|c|c|c|c|c|}
		\hline
		$t$ & $p$ & Order $tp$ & $\dot E_{ISI}(K_p^{(t)})$ & $\dot E_{ISI}(K_p^{(t)}-e)$ & Difference\\
		\hline
		$3$&$2$&$6$&$16.000000$&$15.816551$&$-0.183449$\\
		\hline
		$3$&$3$&$9$&$36.000000$&$37.512601$&$1.512601$\\
		\hline
		$3$&$4$&$12$&$64.000000$&$67.309689$&$3.309689$\\
		\hline
		$4$&$2$&$8$&$36.000000$&$36.654775$&$0.654775$\\
		\hline
		$4$&$3$&$12$&$81.000000$&$84.314847$&$3.314847$\\
		\hline
		$5$&$2$&$10$&$64.000000$&$65.642115$&$1.642115$\\
		\hline
		$6$&$2$&$12$&$100.000000$&$102.706479$&$2.706479$\\
		\hline
	\end{tabular}
	\caption{Numerical verification of Theorem~\ref{thm:ISI-multipartite-energy-change-corrected}.}
	\label{tab:ISI-multipartite-edge-delete}
\end{table}

Table \ref{tab:ISI-multipartite-edge-delete} gives the numerical validation of Theorem \ref{thm:ISI-multipartite-energy-change-corrected}. The only decreasing case is $K_{2,2,2}-e$.
\subsection{Crown and crown multipartite graphs}

Let $C_{p,p}$ be obtained from $K_{p,p}$ by deleting a perfect matching, Figure \ref{fig:crown33} shows $C_{3,3}=K_{3,3}-3K_2$. Every vertex has degree $d=p-1$. Let
$
c=\phi(d,d).
$
Then
$$
\dot A(C_{p,p})
=
\begin{pmatrix}
	0_p&c(J_p-I_p)\\
	c(J_p-I_p)&0_p
\end{pmatrix}.
$$
Next, we discuss its weighted energy.
\begin{proposition}\label{prop:crown-bipartite-corrected}
	For $p\geq2$,
	$$
	\spec(\dot A(C_{p,p}))
	=
	\Big\{(p-1)c,\;-(p-1)c,\;c^{[p-1]},\;(-c)^{[p-1]}\Big\},
	$$
	and therefore
	$
	\dot E_\phi(C_{p,p})=4(p-1)c.
	$
\end{proposition}

\begin{proof}
	The eigenvalues of $J_p-I_p$ are $p-1$ once and $-1$ with multiplicity $p-1$. For a symmetric matrix $B$, the block matrix
	$$
	\begin{pmatrix}0&B\\B&0\end{pmatrix}
	$$
	has eigenvalues $\pm\lambda_i(B)$. Substituting $B=c(J_p-I_p)$ gives the spectrum. Summing absolute values yields the claimed energy.
\end{proof}

\begin{figure}[H]
	\centering
	\begin{tikzpicture}[scale=1.0,every node/.style={circle,draw,minimum size=6mm,inner sep=1pt}]
		\node (u1) at (-2,1.5) {$u_1$};\node (u2) at (-2,0) {$u_2$};\node (u3) at (-2,-1.5) {$u_3$};
		\node (v1) at (2,1.5) {$v_1$};\node (v2) at (2,0) {$v_2$};\node (v3) at (2,-1.5) {$v_3$};
		\draw (u1)--(v2);\draw (u1)--(v3);
		\draw (u2)--(v1);\draw (u2)--(v3);
		\draw (u3)--(v1);\draw (u3)--(v2);
		\draw[dashed] (u1)--(v1);\draw[dashed] (u2)--(v2);\draw[dashed] (u3)--(v3);
		\node[draw=none,rectangle] at (0,-2.3) {dashed edges form the deleted perfect matching};
	\end{tikzpicture}
	\caption{The crown graph $C_{3,3}=K_{3,3}-3K_2$.}
	\label{fig:crown33}
\end{figure}

We now define the multipartite version unambiguously. Let
$
V_i=\{u_{i1},u_{i2},\ldots,u_{ip}\},~ 1\leq i\leq t,
$
be the partite sets of $K_p^{(t)}$. Delete every edge $u_{ik}u_{jk}$ with $i\neq j$ and the same second index $k$. The resulting graph is denoted by $C_p^{(t)}$. Each vertex has degree
$
d=(t-1)(p-1).
$
With $c=\phi(d,d)$, its weighted adjacency matrix has the Kronecker representation
$$
\dot A(C_p^{(t)})
=c\,(J_t-I_t)\otimes(J_p-I_p).
$$
This observation gives the whole spectrum at once.

\begin{theorem}\label{thm:crown-multipartite-spectrum-corrected}
	Let $t,p\geq2$, let
	$
	d=(t-1)(p-1),
	$
	and put
	$
	c=\phi(d,d).
	$
	Then
	$
	\spec(\dot A(C_p^{(t)}))
	=
	\Big\{
	((t-1)(p-1)c)^{[1]},
	(-(t-1)c)^{[p-1]},
	(-(p-1)c)^{[t-1]},
	c^{[(t-1)(p-1)]}
	\Big\}.
	$
	Consequently,
	$
	\dot E_\phi(C_p^{(t)})
	=
	4(t-1)(p-1)|c|.
	$
	In particular, if $c>0$, then
	$
	\dot E_\phi(C_p^{(t)})
	=
	4(t-1)(p-1)c.
	$ 
	If $c\neq0$, then $\dot A(C_p^{(t)})$ is nonsingular. Moreover, if $c>0$, its inertia is
	$$
	\operatorname{In}(\dot A(C_p^{(t)}))
	=
	\left(
	1+(t-1)(p-1),\,
	p+t-2,\,
	0
	\right).
	$$
	If $c<0$, then
	$$
	\operatorname{In}(\dot A(C_p^{(t)}))
	=
	\left(
	p+t-2,\,
	1+(t-1)(p-1),\,
	0
	\right).
	$$ 
\end{theorem}

\begin{proof}
	Let the partite sets of $C_p^{(t)}$ be
	$
	V_i=\{u_{i1},u_{i2},\ldots,u_{ip}\},~ i=1,\ldots,t.
	$
	The graph $C_p^{(t)}$ is obtained from the complete $t$-partite graph
	$
	K_{\underbrace{p,\ldots,p}_{t}}
	$
	by removing, for every fixed $k\in\{1,\ldots,p\}$, all edges joining vertices
	$u_{ik}$ and $u_{jk},~ i\neq j.$
	Thus a vertex $u_{ik}$ is adjacent to the vertices $u_{j\ell}$ precisely when
	$
	j\neq i
	$ and $
	\ell\neq k.
	$
	 
	For every one of the $t-1$ partite sets different from $V_i$, there are exactly $p-1$ allowable choices of the second index $\ell$. Hence each vertex has degree
	$
	d=(t-1)(p-1).
	$
	The graph is therefore regular. Since every edge has endpoints of degree $d$, every nonzero entry of the degree-based weighted adjacency matrix has the same value
	$
	c=\phi(d,d).
	$
	
	We order the vertices part by part as
	$
	V_1,V_2,\ldots,V_t.
	$
	There are no edges within a partite set, so every diagonal $p\times p$ block of $\dot A(C_p^{(t)})$ is $0_p$. Consider now two distinct partite sets $V_i$ and $V_j$. A vertex $u_{ik}$ is adjacent to every vertex of $V_j$ except $u_{jk}$. Thus, relative to the natural ordering of the vertices in the two parts, the $(i,j)$ block is
	$
	c(J_p-I_p).
	$
	Consequently,
	$$
	\dot A(C_p^{(t)})
	=
	\begin{pmatrix}
		0_p & c(J_p-I_p) & \cdots & c(J_p-I_p)\\
		c(J_p-I_p) & 0_p & \cdots & c(J_p-I_p)\\
		\vdots & \vdots & \ddots & \vdots\\
		c(J_p-I_p) & c(J_p-I_p) & \cdots & 0_p
	\end{pmatrix}.
	$$
	 The block pattern is exactly that of $J_t-I_t$, and hence
	$$
	\dot A(C_p^{(t)})
	=
	c(J_t-I_t)\otimes(J_p-I_p).
	$$
	 The matrix $J_t$ has eigenvalue $t$ corresponding to the all-ones vector $\mathbf 1_t$, and eigenvalue $0$ on the orthogonal complement $\mathbf 1_t^\perp$. Therefore
	$
	\spec(J_t-I_t)
	=
	\{(t-1)^{[1]},(-1)^{[t-1]}\}.
	$
	Similarly,
	$
	\spec(J_p-I_p)
	=
	\{(p-1)^{[1]},(-1)^{[p-1]}\}.
	$
	
	To make the multiplicities transparent, let
	$
	x_0=\tfrac{1}{\sqrt t}\mathbf 1_t
	$
	and choose an orthonormal basis
	$
	x_1,\ldots,x_{t-1}
	$
	of $\mathbf 1_t^\perp$. Then
	$
	(J_t-I_t)x_0=(t-1)x_0
	$
	and
	$
	(J_t-I_t)x_i=-x_i,
	$ for $ i=1,\ldots,t-1.
	$
	Likewise, let
	$
	y_0=\tfrac{1}{\sqrt p}\mathbf 1_p
	$
	and choose an orthonormal basis
	$
	y_1,\ldots,y_{p-1}
	$
	of $\mathbf 1_p^\perp$. Then
	$
	(J_p-I_p)y_0=(p-1)y_0
	$
	and
	$
	(J_p-I_p)y_j=-y_j,$ for $j=1,\ldots,p-1.
	$
	
	For arbitrary vectors $x$ and $y$, the Kronecker-product identity
	$$
	(A\otimes B)(x\otimes y)
	=
	(Ax)\otimes(By)
	$$
	implies that if $x$ and $y$ are eigenvectors of $A$ and $B$ with respective eigenvalues $\alpha$ and $\beta$, then $x\otimes y$ is an eigenvector of $A\otimes B$ with eigenvalue $\alpha\beta$. Applying this to
	$
	(J_t-I_t)\otimes(J_p-I_p),
	$
	we obtain four types of eigenvectors.
	
	First,
	$
	x_0\otimes y_0
	$
	has eigenvalue
	$
	(t-1)(p-1).
	$
	After multiplication by $c$, this yields the simple eigenvalue
	$
	(t-1)(p-1)c.
	$
	
	Second, for each $j=1,\ldots,p-1$,
	$
	x_0\otimes y_j
	$
	has eigenvalue
	$
	(t-1)(-1)=-(t-1).
	$
	Thus
	$
	-(t-1)c
	$
	occurs with multiplicity $p-1$.
	
	Third, for each $i=1,\ldots,t-1$,
	$
	x_i\otimes y_0
	$
	has eigenvalue
	$
	(-1)(p-1)=-(p-1).
	$
	Hence
	$
	-(p-1)c
	$
	has multiplicity $t-1$.
	
	Finally, for
	$
	1\leq i\leq t-1,
	$ and $
	1\leq j\leq p-1,
	$
	the vector
	$
	x_i\otimes y_j
	$
	has eigenvalue
	$
	(-1)(-1)=1.
	$
	Therefore
	$
	c
	$
	occurs with multiplicity
	$
	(t-1)(p-1).
	$	
	The total multiplicity is
	$
	1+(p-1)+(t-1)+(t-1)(p-1) =tp,
	$
	which is exactly the order of $C_p^{(t)}$. Thus no further eigenvalues occur, and
	$$
	\spec(\dot A(C_p^{(t)}))
	=
	\Big\{
	((t-1)(p-1)c)^{[1]},
	(-(t-1)c)^{[p-1]},
	(-(p-1)c)^{[t-1]},
	c^{[(t-1)(p-1)]}
	\Big\}.
	$$
	Clearly, 
	$$
	\begin{aligned}
		\dot E_\phi(C_p^{(t)})
		={}&
		(t-1)(p-1)|c|
		+
		(p-1)(t-1)|c|\\
		&+
		(t-1)(p-1)|c|
		+
		(t-1)(p-1)|c|=4(t-1)(p-1)|c|.
	\end{aligned}
	$$ 
	If $\phi(d,d)>0$, as happens for the standard positive degree-based weights, then $c>0$, and this reduces to
	$
	\dot E_\phi(C_p^{(t)})
	=
	4(t-1)(p-1)c.
	$
	
	Finally, we determine nonsingularity and inertia. Since $t,p\geq2$, all three numbers
	$
	t-1,~ p-1,$ and $ (t-1)(p-1)
	$
	are strictly positive. Hence, if $c\neq0$, none of
	$
	(t-1)(p-1)c,~
	-(t-1)c,~
	-(p-1)c,$ and $	c$
	is zero. It follows that
	$
	0\notin\spec(\dot A(C_p^{(t)})),
	$
	and therefore $\dot A(C_p^{(t)})$ is nonsingular.
	
	Assume first that $c>0$. Then the positive eigenvalues are
	$
	(t-1)(p-1)c
	$
	with multiplicity $1$, together with
	$
	c
	$
	with multiplicity $(t-1)(p-1)$. Hence the number of positive eigenvalues is
	$
	n_+
	=
	1+(t-1)(p-1).
	$
	The negative eigenvalues are
	$
	-(t-1)c
	$
	with multiplicity $p-1$ and
	$
	-(p-1)c
	$
	with multiplicity $t-1$. Thus
	$
	n_-
	=
	(p-1)+(t-1)
	=
	p+t-2.
	$
	There are no zero eigenvalues, so
	$
	n_0=0.
	$
	Consequently,
	$$
	\operatorname{In}(\dot A(C_p^{(t)}))
	=
	\left(
	1+(t-1)(p-1),\,
	p+t-2,\,
	0
	\right).
	$$
	
	If $c<0$, multiplication of the matrix by the negative scalar $c$ reverses the sign of every nonzero eigenvalue. Hence the positive and negative counts are interchanged, giving
	$$
	\operatorname{In}(\dot A(C_p^{(t)}))
	=
	\left(
	p+t-2,\,
	1+(t-1)(p-1),\,
	0
	\right).
	$$ 
\end{proof}

The following is an immediate consequence of the above theorem.
\begin{corollary}\label{cor:crown-integrality-corrected}
	For $t,p\geq2$, the weighted matrix $\dot A(C_p^{(t)})$ is integral if and only if
	$
	\phi(d,d)\in\mathbb Z,$ for $d=(t-1)(p-1).
	$
	In particular:
	\begin{enumerate}
		\item the adjacency, $AG$, $GA$, first Zagreb, and second Zagreb spectra are always integral;
		\item the $ISI$ spectrum is integral if and only if $d$ is even;
		\item the Randi\`c spectrum with $\phi(x,y)=1/\sqrt{xy}$ is integral if and only if $d=1$;
		\item the Sombor spectrum is nonintegral for every $d\geq1$ because $\phi(d,d)=d\sqrt2$.
	\end{enumerate}
\end{corollary}

\begin{proof}
	The value $c$ itself occurs as an eigenvalue with multiplicity $(t-1)(p-1)$, so integrality of the spectrum forces $c\in\mathbb Z$. Conversely, if $c$ is an integer, all four eigenvalue types in Theorem~\ref{thm:crown-multipartite-spectrum-corrected} are integers. The listed special cases follow by substituting
	$$
	\phi(d,d)=1,~1,~1,~2d,~ d^2,~\frac d2,~\frac1d,~ d\sqrt2,
	$$
	respectively.
\end{proof}

\begin{table}[H]
	\centering
	\small
	\begin{tabular}{|c|c|c|c|c|c|}
		\hline
		Graph/weight & $d$ & $c=\phi(d,d)$ & Distinct weighted eigenvalues & Energy & Integral?\\
		\hline
		$C_{3,3}$, $A$ & $2$ & $1$ & $\pm2,\pm1$ & $8$ & Yes\\
		\hline
		$C_{3,3}$, $ISI$ & $2$ & $1$ & $\pm2,\pm1$ & $8$ & Yes\\
		\hline
		$C_{4}^{(3)}$, $A$ & $6$ & $1$ & $6,-2,-3,1$ & $24$ & Yes\\
		\hline
		$C_{4}^{(3)}$, $ISI$ & $6$ & $3$ & $18,-6,-9,3$ & $72$ & Yes\\
		\hline
		$C_{4}^{(3)}$, $R$ & $6$ & $\tfrac16$ & $1,-\tfrac13,-\tfrac12,\tfrac16$ & $4$ & No\\
		\hline
		$C_{4}^{(3)}$, $S$ & $6$ & $6\sqrt2$ & $36\sqrt2,-12\sqrt2,-18\sqrt2,6\sqrt2$ & $144\sqrt2$ & No\\
		\hline
	\end{tabular}
	\caption{Numerical and exact checks for the crown formulas.}
	\label{tab:crown-numerics}
\end{table}
Table \ref{tab:crown-numerics} shows numerical values for crown multipartite graphs. The multiplicities are given by Theorem~\ref{thm:crown-multipartite-spectrum-corrected}.

\subsection{Adding an edge to a star}
Theorem 10 of Bilal and Munir \cite{bilalpak} related to $ISI$ energy of star graph does not make sense, since deleting an edge is not possible as we get a disconnected graph. Rather, the author deleted an edge and they omitted the isolated vertex, thereby it is not logical to compare the $ISI$ energy between two graphs of different order.   It is therefore cleaner to compare the star $(S_{n}\cong K_{1,n-1})$ with the connected unicyclic graph ($S_{n}^{+}$) obtained by adding an edge between two leaves. Let
$
S_n=K_{1,n-1}
$
and let $S_n^+$ be obtained by joining two leaves, where $n\geq4$.

For the $ISI$ matrix, each edge of $S_n$ has weight $\tfrac{n-1}{n}$, and hence
$$
\spec(ISI(S_n))
=
\left\{\frac{(n-1)^{3/2}}{n},\;0^{[n-2]},\;-\frac{(n-1)^{3/2}}{n}\right\}.
$$
Thus
$$
\dot E_{ISI}(S_n)=\frac{2(n-1)^{3/2}}{n}.
$$

In $S_n^+$ the two joined leaves have degree $2$, the other $n-3$ leaves have degree $1$, and the center has degree $n-1$, see Figure \ref{fig:star-plus}.
\begin{figure}[H]
	\centering
	\begin{tikzpicture}[scale=1.05,every node/.style={circle,draw,minimum size=6mm,inner sep=1pt}]
		\node (c) at (0,0) {$z$};
		\node (a) at (-2,1.6) {$u$};\node (b) at (2,1.6) {$v$};
		\node (l1) at (-2,-1.5) {$w_1$};\node (l2) at (0,-2.1) {$\cdots$};\node (l3) at (2,-1.5) {$w_{n-3}$};
		\foreach \x in {a,b,l1,l2,l3}{\draw (c)--(\x);}
		\draw[thick] (a)--node[above,draw=none,rectangle] {added edge} (b);
	\end{tikzpicture}
	\caption{The unicyclic graph $S_n^+$.}
	\label{fig:star-plus}
\end{figure}
Let
$$
\alpha=\frac{2(n-1)}{n+1},
\qquad
\beta=\frac{n-1}{n}.
$$
With the center first, the two joined leaves next, and the remaining leaves last,
$$
ISI(S_n^+)=
\begin{pmatrix}
	0&\alpha&\alpha&\beta J_{1\times(n-3)}\\
	\alpha&0&1&0\\
	\alpha&1&0&0\\
	\beta J_{(n-3)\times1}&0&0&0_{n-3}
\end{pmatrix}.
$$

Next, we discuss its $ISI$ spectrum.
\begin{theorem}\label{thm:star-plus-correct-spectrum}
	For $n\geq4$,
	$
	\spec(ISI(S_n^+))
	=
	\Big\{0^{[n-4]},\;-1,\;\theta_1,\theta_2,\theta_3\Big\},
	$
	where $\theta_1,\theta_2,\theta_3$ are the zeros of
	$$
	q_n(x)
	=x^3-x^2
	-\frac{(n-1)^3(n^2+8n+3)}{n^2(n+1)^2}x
	+\frac{(n-3)(n-1)^2}{n^2}.
	$$
	These roots satisfy
	$
	\theta_1>\theta_2>0>\theta_3,
	$
	and
	$
	\dot E_{ISI}(S_n^+)=2(1-\theta_3).
	$
\end{theorem}

\begin{proof}
	The $n-3$ ordinary leaves have identical rows, giving $0$ with multiplicity $n-4$. The vector that is $1$ on one of the two joined leaves, $-1$ on the other, and $0$ elsewhere is an eigenvector with eigenvalue $-1$. 
	
	On vectors that are constant on the two joined leaves and constant on the remaining $n-3$ leaves, the quotient matrix is
	$$
	Q_n=
	\begin{pmatrix}
		0&\frac{4(n-1)}{n+1}&\frac{(n-1)(n-3)}{n}\\
		\frac{2(n-1)}{n+1}&1&0\\
		\frac{n-1}{n}&0&0
	\end{pmatrix}.
	$$
	A direct determinant calculation gives
	$
	\det(xI-Q_n)=q_n(x).
	$
	Since $Q_n$ is similar to a real symmetric matrix, all three roots are real. Its constant term is positive, so the product $\theta_1\theta_2\theta_3$ is negative. The Perron root is positive, and therefore exactly one root is negative. Since the sum of the roots is $1$, the energy is
	$$
	\dot E_{ISI}(S_n^+)=1+\theta_1+\theta_2-\theta_3
	=1+(1-\theta_3)-\theta_3
	=2(1-\theta_3).
	$$
\end{proof}

The $ISI$ energy of $S_n^+$ exceeds that of $S_n.$

\begin{theorem}\label{thm:star-plus-energy-increase}
	For every $n\geq4$,
	$$
	\dot E_{ISI}(S_n^+)>
	\dot E_{ISI}(S_n).
	$$
	Equivalently, deleting the added leaf--leaf edge from $S_n^+$ strictly decreases the $ISI$ energy.
\end{theorem}

\begin{proof}
	Let
	$
	s_n=\tfrac{(n-1)^{3/2}}{n}
	$
	and let $\theta_3<0$ be the unique negative zero of $q_n$. By Theorem~\ref{thm:star-plus-correct-spectrum}, it is enough to prove
	$
	1-\theta_3>s_n,
	$
	or equivalently
	$
	\theta_3<1-s_n.
	$
	Let $u=\sqrt{n-1}$, so $u\geq\sqrt3$ and $n=u^2+1$. By direct substitution, we obtain
	$$
	q_n(1-s_n)
	=
	\frac{u^3H(u)}{(u^2+1)^3(u^2+2)^2},
	$$
	where
	$$
	H(u)=2u^9+5u^8+2u^7+2u^6-8u^5-13u^4-16u^3-12u^2-8u-4.
	$$
	Writing $u=\sqrt3+v$ with $v\geq0$ yields
	$$
	\begin{aligned}
		H(\sqrt3+v)={}&2v^9+(5+18\sqrt3)v^8+(218+40\sqrt3)v^7+(422+518\sqrt3)v^6+(2386+852\sqrt3)v^5\\
		&+(3227+2438\sqrt3)v^4+(4910+2588\sqrt3)v^3+(3804+2034\sqrt3)v^2\\
		&+(1324+1008\sqrt3)v +302+88\sqrt3>0.
	\end{aligned}
	$$
	Hence $q_n(1-s_n)>0$. Since $q_n(x)\to-\infty$ as $x\to-\infty$, $q_n(0)>0$, and $q_n$ has exactly one negative zero, we obtain
	$
	\theta_3<1-s_n<0.
	$
	Therefore
	$$
	\dot E_{ISI}(S_n^+)=2(1-\theta_3)>2s_n=\dot E_{ISI}(S_n).
	$$
\end{proof}

\begin{table}[H]
	\centering
	\small
	\begin{tabular}{|c|c|c|c|c|}
		\hline
		$n$ & $\dot E_{ISI}(S_n)$ & $\dot E_{ISI}(S_n^+)$ & Difference & $\theta_3$\\
		\hline
		$4$ & $2.598076$ & $5.031666$ & $2.433590$ & $-1.515833$\\
		\hline
		$5$ & $3.200000$ & $5.799705$ & $2.599705$ & $-1.899853$\\
		\hline
		$11$ & $5.749596$ & $8.595279$ & $2.845683$ & $-3.297640$\\
		\hline
		$17$ & $7.529412$ & $10.367309$ & $2.837897$ & $-4.183654$\\
		\hline
		$26$ & $9.615385$ & $12.395550$ & $2.780165$ & $-5.197775$\\
		\hline
		$33$ & $10.970869$ & $13.706725$ & $2.735856$ & $-5.853363$\\
		\hline
		$63$ & $15.498047$ & $18.098260$ & $2.600213$ & $-8.049130$\\
		\hline
		$100$ & $19.700751$ & $22.205317$ & $2.504566$ & $-10.102659$\\
		\hline
	\end{tabular}
	\caption{Numerical verification of Theorem~\ref{thm:star-plus-energy-increase}. The values are computed from the exact cubic in Theorem~\ref{thm:star-plus-correct-spectrum}.}
	\label{tab:star-plus-corrected}
\end{table}

Table \ref{tab:star-plus-corrected} gives the numerical values of $\dot E_{ISI}(S_n)$ and  $\dot E_{ISI}(S_n^+)$.

\section{Conclusion and future directions}\label{section 6}

The degree-based weighted adjacency framework provides a common language for a large collection of matrices arising from vertex-degree-based topological indices. For complete multipartite graphs, the decisive reduction is the diagonally symmetrized equitable quotient. It gives the characteristic polynomial, rank, nullity, spectral radius, energy, and integrality information in a compact form and also shows that the natural condition behind three distinct weighted eigenvalues is constancy of the off-diagonal entries of the symmetric quotient, rather than regularity by itself. This distinction produces genuine irregular three-eigenvalue families for familiar weights such as $ISI$ and $M2$. 

Results of energy change due to edge deletion in specific graph families are presented in the paper along with the weighted adjacency matrices that go with them. For almost all weighted adjacency matrices, the energy of the complete graph increases by edge deletion, correcting a finding from \cite{bilalpak}.
Additionally, it is demonstrated that deleting edge from a regular complete multipartite graph increases its energy, which solves problem in \cite{bilalpak}.
Additionally, spectral property results for the crown multipartite graphs are shown. It is challenging to determine if energy increases or decreases when edges are removed from a weighted matrix of a graph.

For graph perturbations, the principal lesson is that a degree-based weighted matrix must be recomputed after an edge is modified. In the complete graph, this leads to an exact threshold separating increase, equality, and decrease of both spectral radius and energy. The same structural viewpoint yields an explicit reduction for one-edge deletion from regular complete multipartite graphs, exact spectra for crown multipartite graphs through a Kronecker product, and a complete $ISI$ energy comparison between $S_n$ and $S_n^+$. These results replace numerical extrapolation by small quotient matrices and exact algebraic criteria.

Several directions remain open. It would be useful to characterize all complete multipartite graphs whose symmetric quotient is singular, to classify irregular families having exactly three or four weighted eigenvalues for the standard degree functions, and to determine integral weighted spectra beyond the regular and crown families. Another natural problem is to extend the one-edge deletion trichotomy from $K_n$ to graphs with two or more degree classes and to identify classes of functions $\phi$ for which energy is monotone under prescribed graph transformations. Recent work on function-weighted spectral radii, perturbations, Li--Feng transformations, and generalized weighted spectral models \cite{LiZheng2025,GaoLiYangZheng2025,LiZheng2026} suggests that quotient methods, eigenvector comparison, and majorization techniques may provide an effective route to these questions.

 \section*{Declarations}
 \noindent \textbf{Data Availability:}	There is no data associated with this article.
 
 \noindent \textbf{Funding:} The authors did not receive support from any organization for the submitted work.
 
 \noindent \textbf{Conflict of interest:} The authors have no competing interests to declare that are relevant to the content of this article.
 


\begin{thebibliography}{0}
	
	
	\bibitem{akbari}
	S. Akbari, E. Ghorbani and M. R. Oboudi,
	Edge addition, singular values, and energy of graphs and matrices,
	\emph{Linear Algebra Appl.} \textbf{430} (2009), 2192--2199.
	
	\bibitem{bilalpak}
	A. Bilal and M. M. Munir,
	ISI energy change due to an edge deletion,
	\emph{Int. J. Quantum Chem.} \textbf{124} (2024), e27501.
	
	\bibitem{braudli}
	R. Brualdi, L. Hogben and B. Shader,
	Energy of a graph,
	in \emph{Notes for the AIM Workshop on Spectra of Families of Matrices Described by Graphs, Digraphs, and Sign Patterns}, 2007.
	
	\bibitem{ping}
	G. Chartrand and P. Zhang,
	\emph{Introduction to Graph Theory},
	Tata McGraw--Hill, New Delhi, 2006.
	
	\bibitem{chen2018}
	X. Chen,
	On $ABC$ eigenvalues and $ABC$ energy,
	\emph{Linear Algebra Appl.} \textbf{544} (2018), 141--157.
	
	\bibitem{chen}
	X. Chen and X. Liu,
	Remarks on the bounds of graph energy in terms of vertex cover number or matching number,
	\emph{Czech. Math. J.} \textbf{71} (2021), 309--319.
	
	\bibitem{cds}
	D. Cvetkovi'c, P. Rowlinson and S. Simi'c,
	\emph{An Introduction to the Theory of Graph Spectra},
	London Mathematical Society Student Texts, Vol. 75,
	Cambridge University Press, Cambridge, 2010.
	
	\bibitem{dasenergy}
	K. C. Das, I. Gutman, I. Milovanovi'c, E. Milovanovi'c and B. Furtula,
	Degree-based energies of graphs,
	\emph{Linear Algebra Appl.} \textbf{554} (2018), 185--204.
	
	\bibitem{dgf}
	K. C. Das, I. Gutman and B. Furtula,
	On spectral radius and energy of extended adjacency matrix of graphs,
	\emph{Appl. Math. Comput.} \textbf{296} (2017), 116--123.
	
	\bibitem{day}
	J. Day and W. So,
	Singular value inequality and graph energy change,
	\emph{Electron. J. Linear Algebra} \textbf{16} (2007), 291--299.
	
	\bibitem{day2}
	J. Day and W. So,
	Graph energy change due to edge deletion,
	\emph{Linear Algebra Appl.} \textbf{428} (2008), 2070--2078.
	
	\bibitem{guo}
	X. Guo and Y. Gao,
	Arithmetic-geometric spectral radius and energy of graphs,
	\emph{MATCH Commun. Math. Comput. Chem.} \textbf{83} (2020), 651--660.
	
	\bibitem{gutmanenergy}
	I. Gutman,
	The energy of a graph,
	\emph{Ber. Math.-Statist. Sekt. Forsch. Graz} \textbf{103} (1978), 1--22.
	
	\bibitem{gutman2021}
	I. Gutman,
	Geometric approach to degree-based topological indices: Sombor indices,
	\emph{MATCH Commun. Math. Comput. Chem.} \textbf{86} (2021), 197--220.
	
	\bibitem{gutman2022}
	I. Gutman, J. Monsalve and J. Rada,
	A relation between a vertex-degree-based topological index and its energy,
	\emph{Linear Algebra Appl.} \textbf{636} (2022), 134--142.
	
	\bibitem{gs}
	I. Gutman and J. Y. Shao,
	The energy change of weighted graphs,
	\emph{Linear Algebra Appl.} \textbf{435} (2011), 2425--2431.
	
	\bibitem{gutmancontributions}
	I. Gutman, I. Red\v{z}epovi'c and J. Rada,
	Relating energy and Sombor energy,
	\emph{Contrib. Math.} \textbf{4} (2021), 41--44.
	
	\bibitem{harary}
	F. Harary and A. J. Schwenk,
	Which graphs have integral spectra?,
	in \emph{Graphs and Combinatorics, Proceedings of the Capital Conference on Graph Theory and Combinatorics},
	Springer, Berlin, 1974, 45--51.
	
	\bibitem{HB}
	H. A. Ganie and B. A. Chat,
	Bounds for the energy of weighted graphs,
	\emph{Discrete Appl. Math.} \textbf{268} (2019), 91--101.
	
	\bibitem{glza}
	M. Ghorbani, X. Li, S. Zangi and N. Amraei,
	On the eigenvalue and energy of extended adjacency matrix,
	\emph{Appl. Math. Comput.} \textbf{397} (2021), 125939.
	
	\bibitem{li3}
	X. Li, Y. Li and J. Song,
	The asymptotic value of graph energy for random graphs with degree-based weights,
	\emph{Discrete Appl. Math.} \textbf{284} (2020), 481--488.
	
	\bibitem{li2}
	X. Li, Y. Li and Z. Wang,
	The asymptotic value of energy for matrices with degree-distance-based entries of random graphs,
	\emph{Linear Algebra Appl.} \textbf{603} (2020), 390--401.
	
	\bibitem{shi}
	X. Li, Y. Shi and I. Gutman,
	\emph{Graph Energy},
	Springer, New York, 2012.
	
	\bibitem{lin1}
	Z. Lin, B. Deng, L. Miao and H. Li,
	On the spectral radius, energy and Estrada index of the arithmetic-geometric matrix of a graph,
	\emph{Discrete Math. Algorithms Appl.} \textbf{14} (2022), 2150108.
	
	\bibitem{nv}
	V. Nikiforov,
	Beyond graph energy: norms of graphs and matrices,
	\emph{Linear Algebra Appl.} \textbf{506} (2016), 82--138.
	
	\bibitem{bilalMatch}
	B. A. Rather and M. Imran,
	Sharp bounds on the Sombor energy of graphs,
	\emph{MATCH Commun. Math. Comput. Chem.} \textbf{88} (2022), 605--624.
	
	\bibitem{bilalMatch1}
	B. A. Rather and M. Imran,
	A note on energy and Sombor energy of graphs,
	\emph{MATCH Commun. Math. Comput. Chem.} \textbf{89} (2023), 467--477.
	
	\bibitem{bilalmdpi}
	B. A. Rather, H. A. Ganie, K. C. Das and Y. Shang,
	General extended adjacency eigenvalues of chain graphs,
	\emph{Mathematics} \textbf{12} (2024), 192.
	
	\bibitem{rodriguez2015}
	J. M. Rodr'iguez and J. M. Sigarreta,
	Spectral study of the geometric-arithmetic index,
	\emph{MATCH Commun. Math. Comput. Chem.} \textbf{74} (2015), 121--135.
	
	\bibitem{wang}
	W.-H. Wang and W. So,
	Graph energy change due to any single edge deletion,
	\emph{Electron. J. Linear Algebra} \textbf{29} (2015), 59--73.
	
	\bibitem{zheng}
	L. Zheng, G. X. Tian and S. Y. Cui,
	On spectral radius and energy of arithmetic-geometric matrix of graphs,
	\emph{MATCH Commun. Math. Comput. Chem.} \textbf{83} (2020), 635--650.
	
	\bibitem{Jamal2023}
	F. Jamal, M. Imran and B. A. Rather,
	On inverse sum indeg energy of graphs,
	\emph{Special Matrices} \textbf{11} (2023), 20220175.
	
	\bibitem{LiYang2023}
	X. Li and N. Yang,
	Some interlacing results on weighted adjacency matrices of graphs with degree-based edge-weights,
	\emph{Discrete Appl. Math.} \textbf{333} (2023), 110--120.
	
	\bibitem{GaoLiYang2024}
	J. Gao, X. Li and N. Yang,
	The effect on the largest eigenvalue of degree-based weighted adjacency matrix by perturbations,
	\emph{Bull. Malays. Math. Sci. Soc.} \textbf{47} (2024), Article 30.
	
	\bibitem{GaoYang2024}
	J. Gao and N. Yang,
	Some bounds on the largest eigenvalue of degree-based weighted adjacency matrix of a graph,
	\emph{Discrete Appl. Math.} \textbf{356} (2024), 21--31.
	
	\bibitem{LiYangLAA2024}
	X. Li and N. Yang,
	Unified approach for spectral properties of weighted adjacency matrices for graphs with degree-based edge-weights,
	\emph{Linear Algebra Appl.} \textbf{696} (2024), 46--67.
	
	\bibitem{LiYangActa2024}
	X. Li and N. Yang,
	Spectral properties and energy of weighted adjacency matrices for graphs with degree-based edge-weight functions,
	\emph{Acta Math. Sin. (Engl. Ser.)} \textbf{40} (2024), 3027--3042.
	
	\bibitem{ShenShan2025}
	C. Shen and H. Shan,
	Extremal spectral radius of degree-based weighted adjacency matrices of graphs with given order and size,
	\emph{Discrete Appl. Math.} \textbf{361} (2025), 315--323.
	
	\bibitem{LiZheng2025}
	X. Li and R. Zheng,
	Extremal results on the spectral radius of function-weighted adjacency matrices,
	\emph{Discrete Appl. Math.} \textbf{373} (2025), 204--216.
	
	\bibitem{GaoLiYangZheng2025}
	J. Gao, X. Li, N. Yang and R. Zheng,
	The Li--Feng transformation of weighted adjacency matrices for graphs with degree-based edge-weights,
	\emph{Discrete Math.} \textbf{348} (2025), 114520.
	
	\bibitem{LiZheng2026}
	X. Li and R. Zheng,
	On the relationship between $({A}^{f})_{\alpha}$-spectral radii of graphs with starlike branch tree or bouquet branch graph and its linear order,
	\emph{Discrete Math.} \textbf{349} (2026), 114734.
	
\end{thebibliography}
\end{document}